\documentclass[11pt]{amsart}
\usepackage{amssymb,amsthm,amsmath}
\usepackage{graphics,psfrag,graphicx,subfigure,epsfig}
\usepackage{enumerate}
\usepackage[all]{xy}
\usepackage{mathtools} 
\usepackage{euscript}
\usepackage{eufrak}
\usepackage{mathrsfs}
\def\f{\mathcal{F}}
\def\fo{\mathcal{F}}
\def\g{\mathcal{G}}
\def\fin{\mathcal{FIN}}
\def\vc{\mathcal{VCY}}
\def\vci{\mathcal{VC}^\infty}
\def\all{\mathcal{ALL}}
\def\evc{\underline{\underline{E}}G}
\def\evce{\underline{\underline{E}}}
\def\efin{\underline{E}G}

\def\nh{N_{G}[H]}
\def\gh{\mathcal{G}[H]}
\def\z{\mathbb{Z}}

\def\r{\mathbb{R}}

\def\h{\mathrm{Diff^+}}
\def\mcg{\Gamma(S)}
\def\mcgm{\Gamma_m(S)}
\def\mcgsi{\Gamma(\widehat{S_i})}
\def\mcgss{\Gamma(\widehat{\ss})}

\def\dif{\mathrm{Diff^+}}
\def\cs{\mathcal{C}(S)}
\def\ss{S_{\sigma}}

\def\C{\mathcal{C}^{\infty}}
\def\gd{\mathrm{gd}}
\def\gdf{\underline{\mathrm{gd}}}
\def\gdvc{\underline{\underline{\mathrm{gd}}}}
\def\rhosm{\rho_{\sigma,m}}
\def\rhos{\rho_{\sigma}}
\def\rhoso{\rho_{\sigma,0}}
\def\mcgsiqi{\Gamma(\widehat{S_i},\mathcal{Q}_i)}

\newtheorem{thm}{Theorem}[section]
\newtheorem{prop}[thm]{Proposition}

\newtheorem*{Main}{Main Theorem}

\newtheorem{proposition}{Proposition}
\newtheorem{theorem}{Theorem}
\theoremstyle{definition}

\newtheorem{defi}[thm]{Definition}

\newtheorem{rem}[thm]{Remark}

\newtheorem{lem}[thm]{Lemma}
\newtheorem{cor}[thm]{Corollary}

\title[Classifying Spaces]{On classifying spaces for the family of virtually cyclic subgroups in mapping class groups}
\author{Daniel Juan-Pineda}
\address{Centro  de  Ciencias Matem\'aticas. \\
UNAM Campus  Morelia\\
Ap.Postal  61-3 Xangari\\ Morelia, Michoac\'an.  M\'EXICO 58089}
 \email{daniel@matmor.unam.mx}
\thanks{We acknolwedge support from research grants from IN105614-DGAPA-UNAM and 151338-CONACyT-M\'exico}

\author{Alejandra Trujillo-Negrete}
\address{Centro  de  Ciencias Matem\'aticas. \\
UNAM Campus  Morelia\\
Ap.Postal  61-3 Xangari\\ Morelia, Michoac\'an.  M\'EXICO 58089}
 \email{aletn@matmor.unam.mx}

\begin{document}

\begin{abstract}
      We give a bound for  the geometric dimension for the family of virtually cyclic groups in mapping class groups of an orientable compact
      surface with punctures, possibly with nonempty boundary and negative Euler characteristic. 
\end{abstract}

\maketitle
\section{Introduction}

Let $S$ be an orientable compact surface with finitely many punctures (possibly zero punctures) and negative Euler characteristic. The mapping class group, $\mcg$, of $S$  is the group of isotopy classes of orientation preserving diffeomorphisms of $S$ that fix point-wise the boundary. Let $m>2$ be an integer and   $\Gamma_m(S)$ be the congruence subgroup of $\mcg$, this is the subgroup of those elements that act trivially on $H_1(S,\z/m)$.

Classifying spaces,  $\efin$, for the family of finite subgroups of a group $G$,    have been extensively studied. For the mapping class group, it is well-known that the Teichm\"uller space $\mathcal{T}(S)$ is a model for $\underline{E} \mcg$ by results of Kerckhoff given in \cite{Kerckhoff}.  Later, J. Aramayona and C. Mart\'inez proved in \cite{aramayona} that the minimal dimension for which there exists a model for  $\underline{E}\mcg$ coincides with the virtual cohomological dimension $vcd(\mcg )$, this has been computed by J. L. Harer   in \cite{harer}.
 
Finite dimensional models for 
 classifying spaces $\evc$ for the family of virtually cyclics  have been constructed for word-hyperbolic groups (Juan-Pineda, Leary \cite{Daniel-Leary}), for groups acting in CAT(0) spaces (Farley \cite{farley}, L\"uck \cite{luck-cat(0)-vyc}, Degrijse and Petrosyan
\cite{Dieter}), and many  other groups.  

A  method to construct a model for $\evc$ is to start with a model for $\efin$ and then try to extend this to obtain a model for $\evc$. In \cite{luck-weiermann}  L\"uck and Weiermann gave a general construction 
using this idea.

The smallest possible dimension of a model of $\evc$  is denoted by $\gdvc G$ and is called the geometric dimension of $G$ for the family of virtually cyclic subgroups.  
We prove that $\gdvc\mcg$  is finite. Degrijse and Petrosyan proved this fact in  \cite{Dieter} for closed surfaces of genus at least 2, although our method produces a larger bound, it is more general since it includes surfaces with boundary and our techniques are different from theirs.

This paper is organized as follows: we review the fundamental material about classifying spaces and mapping class groups in sections   \ref{sec-classif} and \ref{sec-mcg} respectively. In section \ref{sec-commen}, we develop 
the analysis of commensurators  of infinite virtually cyclic subgroups in the mapping class groups and we prove: 

\begin{proposition}\label{intro-prop-comm}
Let $S$ be an orientable closed  surface with finitely many punctures  and $\chi(S)<0$. Let $m\geq 3$ be fixed. Let $C=\langle g\rangle \subset \mcg$ be infinite cyclic and  $n\in \mathbb{N}$ such that $D=\langle g^n\rangle \subset \mcgm$. Then  
$$  N_{\mcg}[C]= N_{\mcg}(D) $$
where $N_{\mcg}[C]$ is the commensurator  of $C$ and   $N_{\mcg}(D)$ is the normalizer of  $D$.  
Furthermore, the subgroup $D$ may be chosen to be  maximal in $\mcgm$. 
\end{proposition}
From Proposition \ref{intro-prop-comm} and a description of normalizers of infinite cyclic subgroups that we develop in section  \ref{sec-commen}, we  prove the following:
\begin{theorem}
Let $S$ be an orientable compact surface with finitely many punctures  and $\chi(S)<0$. Then $\gdvc \mcg< \infty$, that is, the  mapping class group $\mcg$ admits a finite dimensional model for $\underline{\underline{E}} \mcg$. 
\end{theorem}
Our main result is the following:

\begin{Main} 
Let $S$ be an orientable compact surface with finitely many punctures  and $\chi(S)<0$. 
 Let $m\geq 3$, then 
\begin{enumerate}
\item[(1)] $\gdvc \Gamma_m(S) \leq vcd( \Gamma(S)) +1 $;
\item[(2)] Let $[\Gamma(S):\Gamma_m(S) ]$ be the index  of  $\Gamma_m(S)$ in $\Gamma(S)$, then
\begin{align*}\gdvc \Gamma(S) &\leq  [\Gamma(S):\Gamma_m(S) ] \cdot \gdvc \Gamma_m(S)\\ &\leq [\Gamma(S):\Gamma_m(S) ]  \cdot (vcd(\Gamma(S) ) +1 ) .\end{align*} 
\end{enumerate}
Where $vcd(\Gamma(S))$ is the virtual cohomological dimension of $\Gamma(S)$. 
We point out that the bound in $(1)$ is sharp.
\end{Main} 

A key result in finding the above is the following Proposition which we could not find in the literature and we think is one of the most interesting
contributions of this work, see Proposition \ref{prop-modm-maximality}.
\begin{proposition}
Let $S$ be an orientable compact surface with finitely many punctures  and $\chi(S)<0$. 
 Let $m\geq 3$, then the group $\mcgm $ satisfies the  property   that every infinite virtually cyclic subgroup is contained in a unique maximal virtually cyclic subgroup.%$Max_{\vci_{\mcgm}}$. 
 \end{proposition}

 We emphasize that Degrijse and Petrosyan proved in \cite{Dieter} the finiteness of $\gdvc \Gamma(S)$ by other methods. In this work we outline a method that gives a precise description
 of a model for the classifying space for virtually cyclics of mapping class groups. It depends on certain Teichm\"uller spaces and mapping class groups of subsurfaces of $S$.
\vspace{0.3cm}

 Acknowledgment:  A. Trujillo would like to thank John Guaschi  for many fruitful conversations.

%%%%%%%%%%%%%%%%%%%%%%%%%%%Classifying Spaces%%%%%%%
\section{Classifying Spaces for Families} \label{sec-classif}

Let $G$ be a group. A \emph{family} $\fo$ of subgroups of $G$ is a set of subgroups of $G$ which is closed under conjugation and taking subgroups.   The  following are natural examples of families of $G$: \begin{align*}
  \{1\}&=\text{the trivial subgroup} ;\\
  \fin_G &=  \text{finite subgroups of $G$};\\
  \vc_G &=\text{virtually cyclic subgroups of $G$};\\
  \all_G &=  \text{all subgroups of $G$}.
\end{align*}
\begin{defi}Let $\fo$ be a family of subgroups of $G$. A model of the \textit{classifying space $E_{\fo}G$ for the family $\fo$} is a $G$-CW-complex $X$, such that all of its isotropy groups  belong to $\fo$ and if $Y$ is a  $G$-CW-complex  with isotropy groups belonging to $\fo$, there is precisely one $G$-map $Y \rightarrow X$ up to $G$-homotopy.
\end{defi}
In other words, $X$ is a terminal object in the  category of $G$-CW complexes with isotropy groups belonging to $\fo$.
In particular, two models for $E_{\fo}G$ are $G$-homotopy equivalent.
\begin{defi}
Let $G$ be a group, $H\subseteq G$ and $X$  a  $G$-set.  The \emph{$H$-fixed point set}  $X^H$ is  defined as
\begin{align*}
X^H =\{x\in X \mid \text{ for all } h\in H, \; h\cdot x =x  \; \}.
\end{align*}
\end{defi}
\begin{thm} \cite[Thm. 1.9]{luck} \label{thm-esp-clas-defi}
A $G$-CW-complex $X$ is a model of $E_{\fo}G$ if and only if the $H$-fixed point set $X^H$ is contractible for $H \in \fo$ and is empty for $H \not\in \fo$.
\end{thm}
The smallest possible dimension of a model of $E_{\f}G$ is called the \textit{geometric dimension of $G$
 for the family $\f$ } and is usually denoted as $\gd_{\f}G$. When a finite dimensional model of $E_{\f}G$ does not exist, then $\gd_{\f}G=\infty$. \\
We abbreviate $\efin:=E_{\fin_G}G$ and call it the \textit{universal} $G$-CW-\textit{complex for proper $G$-actions}, and we abbreviate $\evc:=E_{\vc_G}G$. Denote by $\gdf G=\gd_{\fin_G}G$ and $\gdvc G=\gd_{\vc_G}G$.  
 It is known that for any groups $H_1,H_2$,
\begin{align}\label{efin-product}
\gdf(H_1\times H_2)\leq \gdf H_1  + \gdf H_2.
\end{align}
For a subgroup $H\subseteq G$, and a  family $\f$ of $G$, denote by
$$\fo \cap H=\{ \text{subgroups of $H$ belonging to }\fo  \}, $$
 the family of subgroups of $H$ induced from $\fo$.
 A model of $E_{\fo\cap H}H$ is given by restricting the action of $G$ to $H$ in  a model of $E_{\fo}G$. Then 
 \begin{align}\label{gd-subg}
 \gd_{\f\cap H}H\leq \gd_{\f}G. 
 \end{align} 

%%%%%%%%%%%%%%%%%%%%%%%%%%%%%%% construction

\subsection{Constructing models from models for smaller families} \label{section-build}

We will use the construction given by  W. L\"uck and M. Weiermann  in  \cite{luck-weiermann}.  In particular for the families $\fin\subset\vc$, 
it is as follows: Let $\vci_G=\vc_G-\fin_G$ be the collection of infinite virtually cyclic subgroups of $G$. Define an equivalence relation, $\sim$, on $ \vci_G$  as
   \begin{equation} \label{rel-eq}
   V\sim W \:\:\Longleftrightarrow  |V \cap W |= \infty ,
   \end{equation}
for $V$ and $W$ in $\vci_G$, where
$|\star|$ denotes the cardinality of the set $\star$.
Let $[\vci_G]$ denote the set of equivalence classes under the above relation
and  let $[H] \in [\vci_G ]$  be the equivalence class of $H$. Define 
\begin{align}\label{nh}
 \nh &=\{ g\in G \mid |g^{-1} H g \cap H|= \infty \},
 \end{align}
 this is the isotropy group of $[H]$ under the $G$-action on $[\vci_G]$ induced by conjugation. Observe that $N_G[H]$ is the commensurator of the subgroup $H \subseteq G$. 
Define a family of subgroups of $\nh$ by
\begin{align} \label{gh}
  \gh:=\{K \in \vci_{\nh}  \mid  |K\cap H|= \infty \} \cup \fin_{ \nh}.
\end{align}
%where $\fin \cap \nh$ consist of all the subgroups of $\nh$ belonging to $\fin$.

The method to build a model of $\evc$ from one of $\efin$ is given  in the following theorem.

\begin{thm}\cite[Thm.2.3]{luck-weiermann} \label{thm-luck-weierm}
Let $\vci_G$ and $\sim$ be as above. Let $I$ be a complete system of representatives, $[H]$, of the $G$-orbits in $[\vci_G]$ under the $G$-action coming from conjugation. Choose arbitrary $N_G[H]$-CW-models for $\underline{E}(N_G[H])$,  $E_{\g[H]}(N_G[H])$ and an arbitrary $G$-CW-model for $\underline{E}(G)$. Define $X$ a $G$-CW-complex by the cellular $G$-pushout
\begin{align*}
\xymatrix{
                \coprod_{[H]\in I} G \times_{N_G[H]} \underline{E}(N_G[H])
                \ar[d]^{\coprod_{[H]\in I} id_G \times_{N_G[H]}f_{[H]}}
                \ar[r]^(0.75){i}
                &\underline{E}(G)
                \ar[d]
                \\
                \coprod_{[H]\in I} G \times_{N_G[H]} E_{\g[H]}(N_G[H])
                \ar[r]
                & X
}
\end{align*}
such that $f_{[H]}$ is a cellular $N_G[H]$-map for every $[H]\in I$ and $i$ is an inclusion of $G$-CW-complexes, or such that every map $f_{[H]}$ is an inclusion of $N_G[H]$-CW-complexes for every $[H]\in I$  and $i$ is a cellular $G$-map. 
Then $X$ is a model for $E_{\g}(G).$
\end{thm}
The maps in Theorem \ref{thm-luck-weierm} are given by the universal property of  classifying spaces for families and  inclusions of families of subgroups. Observe that if $\gdf G$ is finite and both $\gdf N_G[H]$ and $\gd_{\gh} N_G[H]$ are uniformly bounded, then $\gdvc G$ is finite.  
%%%%%%%%%%%%%%
\begin{rem}
Observe that if  $H,K\in \vci_G$ and  both $C_{H}$ and $C_K$ are  infinite cyclic subgroups of $H$ and $K$ respectively, let $\sim$ as in (\ref{rel-eq})
\begin{align}
 H\sim K & \;\;\text{ if only if  }\;\; C_H \sim C_K,   \label{sim-ic}\;\; \text{ and} \\
 \nh \; &=\{ g\in G \mid |g^{-1} H g \cap H|= \infty \}\nonumber \\
 &=\{g\in G \mid |g^{-1}C_H g \cap C_H|=\infty \}. \label{nh=nc}
 \end{align}
Let $\C_G$ be the set of infinite cyclic subgroups $C$ of $G$.  Then the equivalence relation $\sim$ given in (\ref{rel-eq}) can be defined in $\C_G$, we will denote by $[\C_G]$ the set of equivalence classes.
\end{rem}

%For $H\subseteq G$, define   $$W_G H =N_G H /H .$$
%\begin{cor} \cite[Cor. 2.10]{luck-weiermann}   Let $G$ be a group that satisfies property  $Max_{\vci_G}$. Let  $\mathcal{M}$ denote a complete system of representatives of the conjugacy classes of maximal infinite virtually cyclic subgroups $V \subseteq G$. Consider the cellular $G$-pushout
%  \begin{align*}
%\xymatrix{
%      \coprod_{V\in \mathcal{M}} G \times_{N_G V} %\underline{E} N_G V
%      \ar[d]^{\coprod_{V\in \mathcal{M}} id_G \times %f_{V}}
%      \ar[r]^(0.75)i
%      &\underline{E} G
%      \ar[d]
%      \\
%      \coprod_{V\in \mathcal{M}} G \times_{N_G V} E W_G V
%      \ar[r]
%      & X
%}
%\end{align*}
%where $EW_GV$ is viewed as an $N_G V$-CW-complex by   %the projection $N_GV \to W_GV$, the maps starting from the left upper corner are cellular and $i$ is an inclusion  of $G$-CW-complexes.  Then $X$ is a model for $\evc$.
%\end{cor} 
        %%%%%%%%%%%%%%%%%%%%%%%%%%%%%%%%%%%%%%%%%%%
    %%%%%%%%%%%%%%%%%%%%%%%%%%%%%%%%%%%%%%%%%%%%%%%%%%%%%%%%%%%%%%%%%%%%%%%%%%%%%%
\section{Mapping Class groups} \label{sec-mcg}

Let   $S$ be  an orientable compact surface with a finite set $\mathcal{P}$, called punctures, of points removed from
the interior.  We will assume that the surface has negative Euler characteristic. References for this section are  \cite{margalit}, \cite{ivanov} and \cite{Thurston's-work}.

Let  $\h(S,\partial S)$ denote the group of orientation preserving difeomorphisms of $S$ that restrict to the identity on  the boundary  $\partial S$.  We endow this group  with the  compact-open topology.
\begin{defi}
 The \textit{mapping class group of} $S$, denoted $\mcg$, is the group
$$
 \mcg=\pi_0(\dif(S,\partial S)),
$$
that is,  the group of (smooth) isotopy classes of elements of $\dif(S,\partial S)$ where isotopies are required to fix the boundary pointwise.  
\end{defi}

\emph{Congruence subgroups. }  Let $m \in \z$, $m>1$. We  denote by $\mcgm$ the kernel of the natural homomorphism
$$\mcg \to Aut(H_1(S,\z/m\z))$$
defined by the action of diffeomorphisms on the homology group, $\mcgm$ is called the congruence subgroup of $\mcg$.  Note that this subgroup  has
finite index in $\mcg$.\\

\emph{Complex of curves. } An \textit{essential} curve is a simple closed curve of $S$ that is not homotopic
to a point, a puncture, or a boundary component.
The \textit{complex of curves}, denoted by $\cs$, is the abstract simplicial complex  associated to $S$ such that,
(i) Vertices are isotopy classes of essential curves, we denote by $V(S)$ the set of vertices;
(ii) $\cs$ has a $k$-simplex for each $(k+1)$-tuple of vertices, where each pair of corresponding isotopy classes have disjoint representants.  
The \textit{realization} of a simplex is the union of mutually disjoint curves that represent its vertices.
The mapping class group acts on $V(S)$: if $f\in \mcg, \, \alpha \in V(S)$,
the action is given by $f\cdot\alpha =f(\alpha)$. Then $\mcg$ acts on $\cs$, since this action sends simplicies into simplicies.\\

\emph{Dehn twists. }  Let $\alpha, \beta$ be isotopy classes of simple closed curves in $S$. We will denote the Dehn twist about $\alpha$ as $T_{\alpha}$. Let $f\in \mcg$ and $j,k\in \z-\{0\}$. 
We have the following properties  of   Dehn twists, see  \cite[Sec. 3.3]{margalit}, 
\begin{enumerate}
\item $T_{\alpha}^j=T_{\beta}^k$ iff $\alpha=\beta$ and $j=k$;
\item  $fT_{\alpha}^jf^{-1}=T_{f(\alpha)}^j$;
\item  $T_{\alpha}^jT_{\beta}^k=T_{\beta}^kT_{\alpha}^j $ iff $i(\alpha, \beta)=0$. 
\end{enumerate}
%%%%%%%%%%%%%%%%%%%%%%%%%%%%
\subsection{Classification of elements in $\mcg$} \label{sec-prelim}

We will first assume  that  $S$ has empty boundary. 
We assume   the reader  is familiar with the theory of transverse singular foliations. See \cite{margalit} or \cite{Thurston's-work}  for references. \\ 

\emph{Pseudo-Anosov diffeomorphisms. } 
A \textit{diffeomorphism} $\phi\in \dif(S)$  is called \textit{pseudo-Anosov} if there exists a pair of transverse measured foliations $(\f^s,\mu^s)$, $(\f^u,\mu^u)$ of $S$ and a real number $\lambda >1$ such that
 \begin{itemize}
 \item[i.] $\phi(\f^s,\mu^s)=(\f^s, \lambda^{-1}\mu^s); \text{ and } \phi(\f^u,\mu^u)=(\f^u, \lambda \mu^u). $
 \item[ii.] the 1-prongs singularities of these foliations belong to  the set of punctures.
 \end{itemize}
The measure foliation $(\f^s,\mu^s)$ is called the \textit{stable foliation for $\phi$} and $(\f^u,\mu^u)$ is called the \textit{unstable foliation for $\phi$}, and $\lambda$ is the dilatation of $\phi$.
\begin{defi}
An element $f\in \mcg$ is called \textit{pseudo-Anosov} if it is represented by a pseudo-Anosov diffeomorphism. 
\end{defi}

\begin{defi}
  An element $f$ is called \textit{reducible} if  $f$ fixes some simplex of $\cs$  and \textit{irreducible} otherwise.  
\end{defi}
Among the irreducible elements, those of finite order are periodic and those of infinite order are pseudo-Anosov.
There is the following  classification theorem  for elements of the mapping class group,  see  \cite[Thm. 13.2]{margalit}. 
%%%%%%%%%%%%%%%%%%%%   Nielsen-Thurston Classification 
\begin{thm}\textsc{(Nielsen-Thurston classification)} \label{classif-mcg}
  Let  $g,n \geq 0$. Let  S be an orientable surface of genus $g$  and $n$ punctures. 
Each $f\in \mcg$ is either periodic, reducible, or pseudo-Anosov. Further,
pseudo-Anosov mapping classes are neither periodic nor reducible.
A periodic element is represented by a finite order diffeomorphism. 
\end{thm}
\emph{Surfaces with boundary. }
If $ S$ has  non-empty boundary,   we define $f\in \mcg$ to be pseudo-Anosov if $f$ restricts to a pseudo-Anosov diffeomorphism on the punctured surface obtained by removing the boundary $\partial S$. \\

%%%%%%%%%%%%%%%%%%%%%%%%%  inclusion homomorphism

\subsection{Induced homomorphisms from inclusions } \label{sec-homom-cota-k-r}
Let $S$ be an orientable closed surface with finitely many punctures. 
Let $S'$ be an orientable compact subsurface of $S$, the inclusion $S'\to S$ induces a natural homomorphism 
\begin{align} \label{homom-injection}
  \eta \colon \Gamma(S') \to \mcg,
\end{align}
let $g\in \Gamma(S')$ and $\psi\in \h(S',\partial S')$ be a representative difeomorphism of $g$. Then $\eta(g)$ is defined as the isotopy class of the difeomorphism which coincides with $\psi$ in $S'$ and is the identity in $S-S'$.  
\begin{figure}
\begin{center}
\includegraphics[scale=0.35]{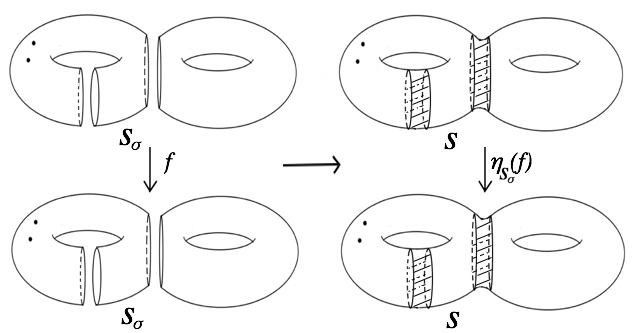}
\caption{The homomorphism $\eta_{S_{\sigma}}(f)$ is the identity  on the  shaded regions and coincides with $f$ on the unshaded regions. }
\label{fig-homom-inclusion}
\end{center}
\end{figure}

Let  $\sigma \in \cs$ with vertices $\alpha_1,...,\alpha_r$ and  $C$ be its realization  in $S$.  Let $N_{\sigma}$ be an open regular   neighborhood of $C$ in  $S$ and denote by  $S_{\sigma}=S-N_{\sigma}=S_1\cup \cdots \cup S_k$, where each $S_i$ is a  connected component.
Let $\beta_i$ and $\gamma_i$ denote  the two boundary components of $N_{\sigma}$ that are isotopic to $\alpha_i$ in S.  
\begin{rem}\label{rem-euler-k-r}
Suppose that the surface $S$ has genus $g$ and $n$ punctures. Note that $\chi(S_{\sigma})=\chi(S)$ and $\chi(S_i)\leq -1$,  therefore $k\leq -\chi(S)$ and $r\leq \frac{-3 \chi(S)-n}{2} $, see \cite{margalit} page 249.  
\end{rem}
See \cite[Thm. 3.18]{margalit} as a reference for the following homomorphisms.  \\

\emph{Cutting the surface. }   From the inclusions $\;\ss\to S\;$ and $\;S_i\to S\;$  denote the induced homomorphisms by  
\begin{align}
\eta_{\ss}\colon \Gamma(\ss)\to \mcg, \;\;\;\; \eta_{S_i}\colon \Gamma(S_i)\to \mcg,
\end{align}
  with
$
\ker( \eta_{\ss})=\langle T_{\beta_1}T_{\gamma_1}^{-1},...,T_{\beta_r}T_{\gamma_r}^{-1}\rangle
$.  Observe that $\Gamma(\ss)=\prod_{i=1}^k\Gamma(S_i)$ and any element in the image $\eta_{\ss}(\Gamma(\ss))$ leaves each subsurface $S_i$ of $S$ invariant. \\

\emph{Corking. } We glue a 1-punctured disk in each boundary component  of $S_{\sigma}$, that is, we corked all boundary components, and
denote this surface  by $\widehat{\ss}=\widehat{S_1}\cup \cdots \cup \widehat{S_k}$, where the $\widehat{S_i}$ are the  connected components of $\widehat{\ss}$ (see Figure 
\ref{fig-corking}). 
Note that the surface $\ss$  has the same Euler characteristic as $S$. 

From the inclusion $\ss \to \widehat{\ss}$, the induced homomorphism defined as   (\ref{homom-injection}) is called the \textit{corking homomorphism of $\ss$}, and it is denoted by 
\begin{align}\label{corking-h}
\theta_{\ss}\colon \Gamma(\ss)\to \mcgss,
\end{align}
 with $\ker (\theta_{\ss})=\langle T_{\beta_1},...,T_{\beta_r},T_{\gamma_1},...,T_{\gamma_r}\rangle$. 
 Let $\mathcal{Q}_i$ be the set of  punctures  in $\widehat{S_i}$ coming from boundary components of $S_i$.  Note that  $\theta_{S_i}(\Gamma(S_i))=\mcgsiqi$ is the subgroup of $\Gamma(\widehat{S_i})$ that fixes pointwise all $p\in \mathcal{Q}_i$.    
Since $\Gamma(\ss)=\prod_{i=1}^k\Gamma(S_i)$ and 
 $\theta_{\ss}=\Pi_{i=1}^k\theta_{S_i}$, then 
\begin{align}\label{corking-homom}
\theta_{\ss} \colon \Gamma(\ss)\to  \prod_{i=1}^k    \Gamma(\widehat{S}_i,\mathcal{Q}_i).  
 \end{align} 
Note that $\prod_{i=1}^k \Gamma(\widehat{S}_i,\mathcal{Q}_i)$ can be seen as  the subgroup of $\mcgss$ that  fixes each subsurface $\widehat{S_i}$ and fixes  pointwise the punctures in $\mathcal{Q}=\cup_{i=1}^k \mathcal{Q}_i$. 

%Since  essential  curves are not isotopic to boundary components, we have
% that $\mathcal{C}(S_i) =\mathcal{C}(\widehat{Si})$. Note that for  any $f\in \Gamma(S_i)$,    $f$ is pseudo-Anosov if and only if $\theta_{S_i} (f)$ is
% pseudo-Anosov, and $f$ is
%reducible if only if $\theta_{S_i} (f)$ is reducible.
\begin{figure}
\begin{center}
\includegraphics[scale=0.24]{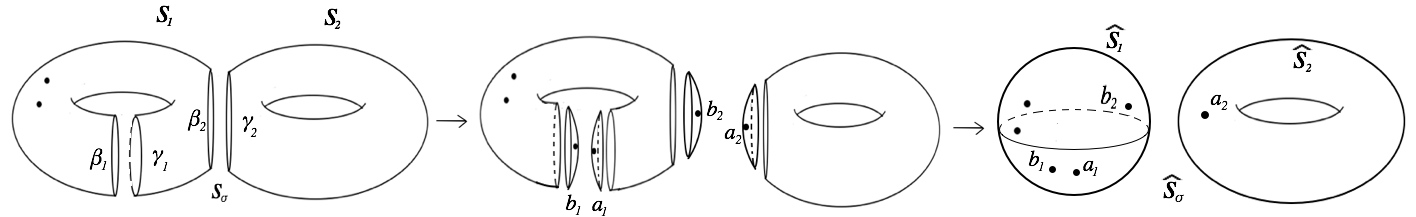}
\caption{Corking each boundary component with a 1-punctured disc. }
\label{fig-corking}
\end{center}
\end{figure}

\subsection{Pure elements }
 
 %%%%%%%%%%%%%%%%%%%%%%%  pure diffeomorphisms
Let $S$ be an orientable compact surface with finitely many punctures.  
A diffeomorphism $\psi \in \h(S)$ is called \emph{pure} if there exists a $1$-submanifold $C$ (possibly empty) of $S$ such that
the following  conditions are satisfied:\begin{enumerate}
 \item[(P)] $C$  is the realization of an element $\sigma\in \cs$ or $C$ is empty; if $C\neq \emptyset$,  $\psi$ fixes $C$, it does not rearrange the components of $S-C$, and it induces on each component of $S-C$ a diffeomorphism isotopic to either a pseudo-Anosov or the identity diffeomorphism.
 \end{enumerate} 
We call an element $f \in \mcg$ \textit{pure} if the isotopy class of $f$ contains a pure diffeomorphism. 
We call $H\subseteq \mcg$ pure if $H$ consists of pure elements.  

%\begin{prop}\label{prop-modm-free-torsion}
%  The group $\mcgm$ is torsion-free, for $m\geq 3$.
%\end{prop}
\begin{thm}\cite[Cor. 1.5 and 1.8]{ivanov} \label{thm-modm-pure}
If $m \geq 3$, then $\mcgm$ is a torsion free group and  a pure subgroup of $\mcg$. 
\end{thm}

From now on, we will consider  the subgroup $\mcgm$ with $m \geq 3$.\\

\emph{Canonical reduction system. } Let $G \subseteq \mcgm$.  An isotopy class $\alpha\in V(S)$ is called an essential reduction class for $G$ if the following conditions are satisfied  (i) $g(\alpha)=\alpha$ for all $g\in G$; (ii) if $\beta\in \cs$ with $i(\alpha,\beta)\neq 0$, \footnote{ The geometric intersection number of $\alpha$ and $\beta$  is denoted by $i(\alpha,\beta)$. }  then $h(\beta)\neq \beta$ for some $h\in G$. The set of essential reduction classes of $G$ is a simplex of $\cs$ and it is called 
 \textit{a canonical  reduction system} of $G$ and it is denoted by $\sigma(G)$. In general, for a subgroup $H\subseteq \mcg$ define $\sigma(H):=\sigma(H\cap\mcgm)$  and for  $f\in \mcg$ define $\sigma(f):=\sigma(\langle f\rangle)$. 
\begin {lem}\cite[Sec. 7.2, 7.3]{ivanov}  \label{lem-red-syst} 
Let $f\in \mcg$,  and $G\subseteq \mcg$. \\   
 (i) If $H\unlhd G $ has finite index, then  $\sigma(G)=\sigma(H)$, \\
 (ii)  $\sigma(fGf^{-1})=f \sigma(G)$.
\end{lem}
\begin{lem}\label{gs=s-reduct}
 Let $S$ be an orientable compact  surface  with finitely many  punctures.  If $f,g \in \mcg$ are such that $f^q=gf^pg^{-1}$  for some $p,q \in \z-\{0\}$,  then $g\sigma(f)=\sigma(f)$. 
\end{lem}
\begin{proof}
  The hypothesis entails that  $\sigma(f^q)=\sigma(gf^pg^{-1})$, by Lemma \ref{lem-red-syst} 
  \begin{align*}
  \sigma(f^q)=\sigma(f)\;\;\text{ and } \;\;\sigma(gf^pg^{-1})=\sigma(gfg^{-1}),
  \end{align*} then
  $\sigma(gfg^{-1})=\sigma(f)$
and by Lemma \ref{lem-red-syst}   $\sigma(gfg^{-1})=g\sigma(f)$,  therefore  $g\sigma(f)=\sigma(f)$. 
\end{proof}
%
%\begin{lem} \cite[Cor. 7.12]{ivanov} \label{reduc-sr-ne}
%Reducible elements in $\mcg$  have canonical reduction %system non-empty. 
%\end{lem}

\emph{Canonical Form. } By cutting $S$ along a reduction system $\sigma'$ of an element in $\mcg$ and applying the Nielsen-Thurston classification Theorem to each subsurface of $S_{\sigma'}$ we can obtain a decomposition of the element as follows, see Figure \ref{fig-canonicalf}.   
%%%%%%%%%%%%%%%%%%%%   CANONICAL FORM 
\begin{thm} \cite[Cor. 13.3]{margalit}
\label{canonical-descom-mod}
Let $f\in \mcg $ and $\sigma=\sigma(f)$ be its canonical
reduction system with
vertices $\alpha_1,...,\alpha_r$. Let    
$S_{\sigma}=S_1\cup ...\cup S_k$ be as before and let   $\overline{N_{\sigma}}=S_{k+1}\cup \cdots \cup S_{k+r}$ be the union of pairwise disjoint closed neighborhoods $S_{k+i}$  of curves representatives of the $\alpha_i$.   Then 
there is a representative $\phi $ of $f$ that
permutes the $S_i$, so that some power of $\phi$ leaves invariant each
$S_i$. Moreover, there exists an integer  $p> 0$ so that 
$\phi^p(S_i)=S_i$ for all $i$ and
\begin{align}\label{canonical-descomp}
  f^p=\prod_{i=1}^{k} \eta_{S_i}(\overline{f}_i) \prod_{j=1}^{r } T_{\alpha_j}^{n_j},
\end{align}
where  each $\overline{f_i} \in \Gamma(S_i)$ is either pseudo-Anosov or the identity  and $n_j \in \z$ for  $1\leq j \leq r$.
\end{thm} 
 When $f\in\mcgm$ and $m\geq 3$  the integer $p$ can always be taken to be one.   

\begin{figure}
\begin{center}
\includegraphics[scale=0.3]{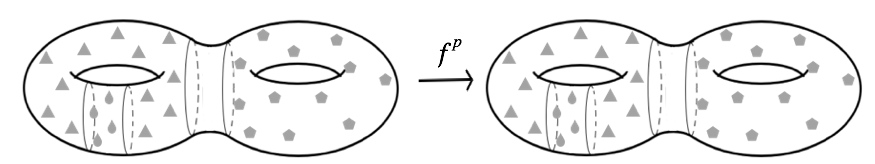}
\caption{The canonical form of $f^k$, where each subsurface is fixed. A shaded region indicates a pseudo-Anosov component or a Dehn Twist and  a unshade region indicates an identity component.    }
\label{fig-canonicalf}
\end{center}
\end{figure}
%

%%%%%%%%%%%%%%%%%%%%%%%   STABILIZERS
%%%%%%%%%%%%%%%%%%%%%%%%%%%%%%%%%%%%%%%%%%%%%%%%%%%%%%%%%%%
\subsection{Stabilizers }\label{section-stabilizers}
\label{sec-stabilizer}
 
In this section, we shall consider $S$ with empty boundary. The stabilizers  $\mcg_{\sigma}$  will be used to understand  the normalizers of reducible elements of  $\mcg$.

 Let  $\sigma \in \cs$ and    $S_{\sigma}=S_1\cup \cdots \cup S_k$ be as in Section \ref{sec-prelim}. 
  Denote   the stabilizer subgroup of $\sigma$  in $\mcg$ by 
       $$\mcg_{\sigma}=\{g \in \mcg | g(\sigma)=\sigma \}.$$
% Let $f \in \mcg_{\sigma}$.  Choose a representative diffeomorphism $F\in \dif(S)$ of $f$ that fixes  $C$, then the diffeomorphism $F$ restricted to $S-C$ determines a diffeomorphism of $S-C$ which  gives an element of $\Gamma(S-C)$.  Since  $S-C$ is  naturally homeomorphic to $\widehat{S_{\sigma}}$, there is a  natural isomorphism   $\Gamma(S-C)\simeq \Gamma(\widehat{S_{\sigma}})$. In this way, we obtain  the following  homomorphism given in \cite[Sec.7.5]{ivanov} and in \cite[Prop. 3.20]{margalit}. 
%%%%%%%%%%%%%%%%%%%%%%%%%%%%%%
%%%%%%%%%%%%%%%%%%%%%%%%%%%%%%%%%%%%  CUTTING HOM
%%%%%%%%%%%%%%%%%%%%%%%%%%%%%%%%
\begin{prop}\cite[Prop. 3.20]{margalit}\label{cutting-homom}
 Let $\sigma \in \cs$ with
vertices $\alpha_1,...,\alpha_r$ and  $\ss=S_1\cup \cdots \cup S_k$. Then there is a well defined homomorphism
\begin{align} 
\rhos \colon \mcg_{\sigma} \to \mcgss.
\end{align}
where  $\ker (\rhos)=\langle T_{\alpha_1},...,T_{\alpha_r}\rangle$  is the  free abelian subgroup generated by Dehn twists about the curves $\alpha_1,..., \alpha_r$.
\end{prop}
Let $\mcg_{\sigma}^0$ be the finite index  subgroup of $\mcg_{\sigma}$  that fixes  each $\alpha_i$ with orientation, since  elements in $\mcg_{\sigma}^0$  are  orientation-preserving, it follows that they also preserve the sides of each curve $\alpha_i$ in $S$,  
thus they fix each subsurface $S_i$. Denote the restriction $\rhoso= \rhos|_{\mcg_{\sigma}^0}$,      
 then
 $\rhoso= \theta_{\ss} \eta_{\ss}^{-1}.$   %Since $\ker(\eta_{\ss})\subset \ker(\theta_{\ss})$, then $\rhoso$ is well defined. 
 Therefore we have   that
\begin{align}\label{cutting-homom-0}
\rhoso\colon \mcg_{\sigma}^0 \to \prod_{i=1}^k \Gamma(\widehat{S}_i,\mathcal{Q}_i),
\end{align}
 is surjective and  $\ker (\rhoso)=\ker(\rhos)$.

\begin{rem}\label{rem-image-pseudo}
By  \cite[Thm. 1.2]{ivanov}, elements in $\mcgm_{\sigma}$ do not rearrange the components of $S_{\sigma}$ and fix each curve of $\sigma$, 
 so $\mcgm_{\sigma}\subseteq \mcg_{\sigma}^0$.
Let  $f\in\mcgm$, $\sigma=\sigma(f)$ with vertices $\alpha_1,...,\alpha_r$, and let the canonical form of $f$ be as follows  
\begin{align*}
  f=\Pi_{i=1}^{k} \eta_{S_i}(\overline{f}_i) \Pi_{j=1}^{r } T_{\alpha_j}^{n_j}. 
\end{align*} 
Moreover, assume  that     $\rhoso(f)=(f_1,...,f_k)$. Note that $f\in \mcgm_{\sigma}$ and 
$$(\overline{f}_1,...,\overline{f}_k) \in \eta_{\ss}^{-1}(f).
$$ 
Since $\rhoso(f)=\theta_{\ss}\eta_{\ss}^{-1}(f)$ and $\rhoso$ is well defined, then 
\begin{align*}
(f_1,...,f_k)&=\theta_{\ss}(\overline{f}_1,...,\overline{f}_k)\\&= (\theta_{S_1}(\overline{f}_1),...,\theta_{S_k}(\overline{f}_k)),
\end{align*}
hence for each $i$, $f_i=\theta_{S_i}(\overline{f}_i)$,  therefore  $f_i$ is pseudo-Anosov or the identity. 
 \end{rem}
%%%%%%%%%
%%%%%%%%%%%%%%%
\section{Commensurators in $\mcg$}\label{sec-commen}
In order to  build a model  $\evce \mcg$ as  mentioned in Sec. \ref{section-build},  we need to know the commensurators of the infinite virtually cyclic subgroups in $\mcg$. 

\subsection{Condition (C) for $\mcg$ } \label{proof-thm-condition-c}
Using the following condition,   we may give a description of the commensurators of infinite virtually cyclic subgroups. 
\begin{enumerate}
\item[(C)] For every $g,h\in G$, with $|h|=\infty$,  and $k,l \in \z$,
$$ \text{if  } \;\;gh^kg^{-1}=h^l  \;\;\text{ then }  \;\;|k|=|l|.$$
\end{enumerate}  

%%%%%%%%%%%%%%%%%%%%%%%%%%%%%%%%%%%%%%%%%%%%%%%%%%%%%%%%%%%%%
\begin{prop}\label{cond-c-mod} Let $S$ be an orientable compact  surface with finitely many punctures  and $\chi(S)<0$. The group $\mcg$ satisfies condition (C).   
\end{prop}
\begin{proof}
\textbf{Case I.}
 The proof when $S$ has empty boundary is given in Propositions  \ref{pAnosov-cond-c} and \ref{reducible-cond-c}, since  elements with infinite order in $\mcg$ are reducible or pseudo-Anosov. \\
\textbf{Case II.} Assume $\partial S\neq \emptyset$  and it has $b$ connected components. Let $f,g\in \mcg$, with $|f|=\infty$ such  that $gf^pg^{-1}=f^q$ for some $p,q\in \z-\{0\}$.  
Let $$\theta_S\colon \mcg \to \Gamma(\widehat{S})$$  be the corking homomorphism of $S$ as in (\ref{corking-h}),   
then  $\ker(\theta_S) \simeq \z^b $ is the free abelian group generated by Dehn twists about curves isotopic to the  boundary components of $S$.  
Since  $\ker( \theta_S)$ is central in $\mcg$,   if $f\in \ker( \theta_S)$ we conclude that $p=q$. \\
On the other hand, if $f\not \in \ker( \theta_S)$, applying $\theta_S$, we have     
 $$\theta_S(g)\theta_S(f)^p\theta_S(g)^{-1}=\theta_S(f)^q,$$
and $|\theta_S(f)|=\infty$, applying Case I  we  conclude that $|p|=|q|$.  
\end{proof}

About  powers of a pseudo-Anosov diffeomorphism  we have the following:
Suppose that $S$ has empty boundary. Let  $f\in \mcg $ be a pseudo-Anosov ele-ment and $\phi $ be a pseudo-Anosov diffeomorphism in its class,  with $(\f^s_{\phi},\mu^s_{\phi})$, $(\f^u_{\phi},\mu^u_{\phi})$  the stable and unstable foliations of ${\phi}$ respectively and $\lambda_{\phi}$  its  dilatation.   For $n\in \z-\{0\}$,
\begin{align}
 \phi^n (\f^s_{\phi},\mu^s_{\phi})=(\f^s_{\phi},\lambda_{\phi}^{-n}\mu^s_{\phi}), \label{power-p-a-fol}\\
 \phi^n (\f^u_{\phi},\mu^u_{\phi})= (\f^u_{\phi},\lambda_{\phi}^{n}\mu^u_{\phi}).\nonumber
\end{align}
The reader may consult   \cite{Thurston's-work} and \cite{McCarthy} as references.

\begin{prop}\label{pAnosov-cond-c} 
 Let $S$ be an orientable closed surface with finitely many punctures  and $\chi(S)<0$. Let $r\in \mathbb{N}$ and $h_1,...,h_r\in \mcg$ be pseudo-Anosov mapping classes, suppose that there exist $g_1,...,g_r\in \mcg$ and a permutation $\gamma\in \Sigma_r$ such that  
\begin{align}\label{eq-conj-pAs}
g_i h_i^p g_i^{-1}=h_{\gamma(i)}^q, \;\;\; \text{ for all } i \in \{1,...,r\},
\end{align}
for  some $p,q\in \mathbb{Z}-\{0\}$, then $|p|=|q|.$
\end{prop}
  \begin{proof}
Suppose that $p>0$ and let $I=\{1,...,r\}$. For each $i \in I$, let $\phi_i$  be pseudo-Anosov diffeomorphisms in the class  $h_i$, for each $i$ there exists  $G_i\in \dif(S)$ in the class $g_i$, such that $G_i\phi_i^pG_i^{-1}=\phi_{\gamma(i)}^q$, this follows by the uniqueness of pseudo-Anosovs \cite[Exp. 12]{Thurston's-work}.    \\ 
Let   $(\f_i^s,\mu_i^s)$ and $(\f_i^u,\mu_i^u)$ be the stable and unstable foliations,  respectively, of $\phi_i$, with dilatation $\lambda_i >1$, $i\in I$. 
By (\ref{power-p-a-fol}), if $n>0 $ (or $n<0$), $(\f^s_{i},\mu^s_{i})$ and  $(\f^u_{i},\mu^u_{i})$ 
are the  stable and unstable   (unstable and stable)  foliations respectively of $\phi_i^n$ with dilatation $\lambda_{i}^n$ (or $\lambda_i^{-n}$).

Suppose that $q>0$. Then  $G_i$ sends  the stable and unstable foliations  of $\phi_i^p $ to the  stable and unstable  foliations  of $\phi_{\gamma(i)}^q $ respectively (\cite[Lem. 16, Exp. 12]{Thurston's-work}). Since the foliations are uniquely ergodic, the measure is up to a constant,  that is,
\begin{align*}
G(\f^s_i,\mu^s_i)=(\f^s_{\gamma(i)},a\mu^s_{\gamma(i)}), &\;\;
G(\f^u_i,\mu^u_i)=(\f^u_{\gamma(i)},b\mu^u_{\gamma(i)}),
\end{align*} 
with $ab=1$. Thus
\begin{align*}
 \phi_{\gamma(i)}^q(\f^s_{\gamma(i)},a\mu_{\gamma(i)}^{s})&= 
 \phi_{\gamma(i)}^q(G(\f^s_{i},\mu^s_{i}))\\
 &=  G \phi_i^p G^{-1}(G(\f^s_{i},\mu^s_{i}))\\
   &= G\phi_i^p (\f^s_{i},\mu^s_{i}) \\
   &= G (\f^s_{i},\lambda_{i}^{-p}\mu^s_{i})  \\
   &=(\f^s_{\phi},\lambda_{i}^{-p}a\mu^s_{\phi}),
 \end{align*}
then   the diffeomorphisms $\phi_i^p$  and $\phi_{\gamma(i)}^q$ have the same dilatation. Then for each $i$, we have    
$\lambda_i^p=\lambda_{\gamma(i)}^q $. Since $\lambda_i>1$, for all $i$, $\lambda_i^p$ has only one real positive $q$-root, then  we can conclude that 
$$\lambda_i^{\frac{p}{q}}=\lambda_{\gamma(i)}, \;\;\; \text{for all }  i.$$ 
If for some $i$, $\gamma(i)=i$, then we conclude that $p=q$. If $\gamma(i)\neq i$ for all $i$, let $n\geq 2$ be the minimum positive integer such that $\gamma^n(1)=1$.  Since  
$\lambda_i^{\frac{p}{q}}=\lambda_{\gamma(i)}$ for all $i$, then 
\begin{align}\label{lambda}
\lambda_{\gamma^a(1)}^{\frac{p}{q}}=\lambda_{\gamma^{a+1}(1)}, \;\: a\in \{0,...,n\},
\end{align}
thus  $\lambda_1=\lambda_1^{(\frac{p}{q})^{n-1}}$. Since $\lambda_1> 1$, we conclude that $p=q$.   
On the other hand, if $q<0$, we have that $\lambda_i^p=\lambda_{\gamma(i)}^{-q}$, and in a similar way, we conclude  $p=-q$. 
 \end{proof}

\begin{prop} \label{reducible-cond-c}
 Let $S$ be an orientable closed  surface with finitely many punctures  and $\chi(S)<0$.  Let $f\in \mcg$ be a  reducible element such that $f^q=gf^pg^{-1}$ for some $g\in \mcg$ and $p,q \in \z-\{0\}$, then $|q|=|p|$.
\end{prop}
\begin{proof} 
Let $\sigma=\sigma(f)$,  suppose that  $\sigma$ has vertices $\alpha_1,...,\alpha_r$.  By hypothesis and by  Lemma \ref{gs=s-reduct},  $g\sigma=\sigma$, so $g$ lies in the stabilizer  $\mcg_{\sigma}$. 
Let $\rhos$ as in Theorem \ref{cutting-homom},
  \begin{align*}
\rhos \colon \mcg_{\sigma }\to \mcgss.
\end{align*}
\textbf{Case 1.} Suppose that $f\in \ker (\rhos) $, that is,   $f=\Pi_{i=1}^{r} T_{\alpha_i}^{n_i} $; without loss of generality we can suppose  that  $gdc\{n_1,...,n_r\}=1$. Note that $g$ may permute   $\alpha_1,...,\alpha_r$, let $\delta\in \Sigma_r$ such that $g(\alpha_i)=\alpha_{\delta(i)}$ for all $i\in I$. Since the $\alpha_i's$ have  disjoint realizations in $S$, the Dehn twists $T_{\alpha_i}$ commute, thus  by hypothesis and because $gT_{\alpha_i}g^{-1}=T_{g(\alpha_i)}=T_{\alpha_{\delta(i)}}$ (see  \cite[Sec. 3.3]{margalit}), we have the following,
\begin{align*}
\Pi_{i=1}^{r} T_{\alpha_i}^{qn_i} &= g\Pi_{i=1}^{r} T_{\alpha_i}^{pn_i}g^{-1}\nonumber \\
&= \Pi_{i=1}^{r} T_{\alpha_{\delta(i)}}^{pn_i}
\end{align*}
 then   we conclude that $qn_i=pn_{\delta^{-1}(i)}$ $\text{ for all } i \in I$,  we can regard these  as vectors $v_1=(n_1,...,n_r)$ and $v_2=(n_{\delta^{-1}(1)},...,n_{\delta^{-1}(r)})$ in $\z^r$ and  $pv_1=qv_2$. Then $v_1$ and $v_2$ are in the same  line of $\r^r$, and since $v_2$ is obtained by a permutation of the coordinates of $v_1$, we conclude that $v_1=v_2$ or $v_1=-v_2$, therefore $|p|=|q|$.   
\\
\textbf{Case 2.} Since $\mcgm \unlhd \mcg $ is of finite index, we may assume that $f\in \mcgm$. 
Suppose that $f \not \in \ker(\rhos)$. Since $f\in \mcgm$, $\rhos(f)$ fixes each subsurface $\widehat{S_j} $ of $\widehat{\ss}$. Let $J=\{1,...,k\}$, for every $j\in J$, let $\tilde{f}_j= \rhos(f)|_{\widehat{S_j}} \colon \widehat{S}_j \to \widehat{S}_j.$ 
Note that $\rhos(g)$ may permute the subsurfaces $\widehat{S_i}$ of $\widehat{\ss}$, let $\tilde{g}_j=\rhos(g)|_{\widehat{S}_j}$ for each $j\in J$ and  let $\gamma\in \Sigma_k$ such that
$\tilde{g}_j \colon \widehat{S}_j \to \widehat{S}_{\gamma(j)}$, $ \text{ for all } j \in J$. 
By hypothesis  we have that  $\rhos(f)^q=\rhos(g) \rhos(f)^p\rhos(g)^{-1}$, then 
\begin{align}\label{conjugate-red}
\tilde{g}_j \tilde{f}_j^p\tilde{g}_j^{-1}=\tilde{f}_{\gamma(j)}^q, & \;\;\; \text{ for all } j\in J.
\end{align}  
Since $f\in \mcgm$ and $\rhos(f)\neq Id$, for some $l\in J$, $\tilde{f}_l$ is pseudo-Anosov (see Remark \ref{rem-image-pseudo}). Let $x>0$  be the minimum positive integer such that $\gamma^x(l)=l$,  from (\ref{conjugate-red}) it follows that 
\begin{align*}
\tilde{g}_{\gamma^i(l)}\tilde{f}_{\gamma^i(l)}^p \tilde{g}_{\gamma^i(l)}^{-1}=\tilde{f}_{\gamma^{i+1}(l)}^q, \;\;\; \text{ for all } i\in \{0,1,...,x\}. 
\end{align*}
Then $\tilde{f}_{\gamma^i(l)}$ is pseudo-Anosov for each $i\in \{0,1,...,x-1\}$.  Observe that $\widehat{S}_l$  and $\widehat{S}_{\gamma^i(l)}$ are homeomorphic for all $i\in\{0,1,...,{x-1}\}$, then we can apply Proposition \ref{pAnosov-cond-c}, therefore $|p|=|q|$.
\end{proof}

\subsection{Description of commensurators}

We denote by $C_G(f)$ and  $N_G(f)$  the centralizer and normalizer respectively  of the subgroup $\langle f\rangle $ in $G$.
\begin{thm} \cite[Thm. 6.1]{bonatti-paris}\label{f^m=g^m}
 Let  $S$ be an orientable compact surface  with finitely many punctures.  Let $G \subseteq \mcg$ be a pure subgroup. If $f,g\in G $ are such that $f^t=g^t$ for some $t\geq 1$, then $f=g$.
\end{thm} 
\begin{lem}\label{lem-n(f)=n(fk)}
 Let  $S$ be an orientable compact surface  with finitely many punctures. Let  $f\in \mcgm$, $t\in \z-\{0\} $ and let $\Gamma$ be either $\mcg$ or $\mcgm$, then  
 $$ C_{\Gamma}(f)=C_{\Gamma}(f^t)\;\; \text{ and }\;\; N_{\Gamma}(f)=N_{\Gamma}(f^t). $$
  %\begin{align*}
  %&C_{\mcgm}(f)=C_{\mcgm}(f^t), \;\;\;\; \;C_{\mcg}(f)=C_{\mcg}(f^t)\\ \text{and} \;\;\;\;  & \\
  %&N_{\mcgm}(f)=N_{\mcgm}(f^t), \;\;\;\; \;N_{\mcg}(f)=N_{\mcg}(f^t) .
  %\end{align*}
\end{lem}
\begin{proof}
  Suppose that $t\geq 1$.   Since $N_{\mcgm}(f)\subseteq N_{\mcgm}(f^t)$, we need to prove that $N_{\mcgm}(f^t)\subseteq N_{\mcgm}(f)$. 
  If $h\in N_{\mcgm}(f^t)$,  we have
  \begin{align*}
    (hfh^{-1})^t =(f^i)^t
   \end{align*}
for some $i \in \{1,-1\}$.   Since $f^i, hfh^{-1}\in \mcgm$ and $\mcgm$ is pure (Thm. \ref{thm-modm-pure}), we can apply  Theorem \ref{f^m=g^m} and we conclude that
  $$hfh^{-1}=f^i,$$
therefore, $h\in N_{\mcgm}(f)$. Thus $\mcgm$ is a  normal subgroup of $\mcg$, we can use Theorem \ref{f^m=g^m} in similar way to prove that $N_{\mcg}(f)=N_{\mcg}(f^t)$.  We have the proof for the centralizers by taking $i=1$. 
\end{proof}
%%%
\begin{lem} \label{lem-norm-dehnt}
Let  $S$ be an orientable compact surface  with finitely many punctures. Let $\sigma\in \cs$ with vertices $\alpha_1$,...,$\alpha_r$, and $f=\Pi_{i=1}^r T_{\alpha_i}^{n_i} $, with $n_i\in \z-\{0\}$, then for any $k\neq 0$, $$N_{\mcg}(f)=N_{\mcg}(f^k).$$ 
\end{lem}
\begin{proof}
Let $g\in N_{\mcg}(f^k)$, we will  prove that $g\in N_{\mcg}(f)$. By Lemma \ref{lem-red-syst}, $g(\sigma)=\sigma$ and $g$ may permute the classes  $\alpha_1,...,\alpha_r$. Let $\delta\in \Sigma_r$ such that $g(\alpha_i)=\alpha_{\delta(i)}$ for all $i$. By the results about Dehn twists given in Section \ref{sec-prelim} and  since $f^{jk}=gf^kg^{-1}$, for some $j\in\{1,-1\}$, we have  
\begin{align*}
\Pi_{i=1}^{r} T_{\alpha_i}^{jkn_i} &= g\Pi_{i=1}^{r} T_{\alpha_i}^{kn_i}g^{-1}\\
&= \Pi_{i=1}^{r} T_{\alpha_{\delta(i)}}^{kn_i},
\end{align*}
 then we  conclude that 
   $jkn_i=kn_{\delta(i)}$ $\text{ for all } i$,  so $jn_i=n_{\delta(i)}$ for all $i$. 
  On the other hand, we have that 
   \begin{align*}
gfg^{-1}&= g\Pi_{i=1}^{r} T_{\alpha_i}^{n_i}g^{-1}\\
& =\Pi_{i=1}^{r} T_{\alpha_{\delta(i)}}^{n_i},
\end{align*}
since $jn_i=n_{\delta(i)}$ then 
 \begin{align*}
gfg^{-1}=\Pi_{i=1}^{r} T_{\alpha_{\delta(i)}}^{jn_{\delta(i)}}=f^{j},
\end{align*}
therefore $g\in N_{\mcg}(f)$. 
\end{proof}

\begin{lem} \cite[Lem. 4.2]{luck-cat(0)-vyc}
 \label{lem-norm-cond-c}
Suppose that $G$ satisfies Condition (C). Then, for any $C\in \C_G$ there is a nested sequence of subgroups
$$N_G(C)\subseteq N_G(2! C) \subseteq N_G(3!C) \subseteq N_{G}(4!C)\subseteq \cdots$$
where $k!C$ is the subgroup of $C$ given by $\{h^{k!}| h\in C\}$,
observe that
$$N_G[C]=\bigcup_{k\geq 1} N_G(k!C).  $$
The subgroup $N_G(K!C)$ denotes the normalizer of $k!C$ in $G$. 
\end{lem} 
We will follow the same notation as Section \ref{section-build}. Let $\C_G$ be the set of infinite cyclic subgroups of $G$. 
\begin{prop}\label{prop-comm-mcgm}
Let $S$ be an orientable closed  surface with finitely many punctures  and $\chi(S)<0$. Let $m\geq 3$ be fixed. Let $C=\langle g\rangle \in \C_{\mcg}$ and  $n\in \mathbb{N}$ such that $g^n\in \mcgm$. Then  
$$N_{\mcg}[C]=N_{\mcg}(g^n). $$ 
Furthermore, the subgroup $\langle g^n\rangle$ can be choosen maximal in $\C_{\mcgm}$. 
\end{prop}
\begin{proof}
Let $C \in \C_{\mcg}$, $[C]$ its class and suppose that $C=\langle g\rangle$. 
From Lemma \ref{lem-norm-cond-c}, we have that
$$N_{\mcg}(g)\subseteq N_{\mcg}(g^{2!})\subseteq N_{\mcg}(g^{3!})\subseteq \cdots, \;\;\; \text{and}$$ 
\begin{align*}
N_{\mcg}[C]=\bigcup_{k\geq 1} N_{\mcg}(g^{k!}).  
\end{align*} 
Let $n\in \mathbb{N}$  such that $g^n\in \mcgm$, 
by  Lemma \ref{lem-n(f)=n(fk)},   we have that 
$$N_{\mcg}(g)\subseteq N_{\mcg}(g^{2!})\subseteq \cdots \subseteq N_{\mcg}(g^{n!})=N_{\mcg}(g^{(n+k)!}) $$
for any $k\geq1$, and $N_{\mcg}(g^{n!})=N_{\mcg}(g^{n})$,  then   $$N_{\mcg}[C]=N_{\mcg}(g^n).  $$ 
We will prove in Proposition \ref{prop-modm-maximality} that $\langle g^n\rangle$ is contained in a unique  maximal $\overline{C}\in \C_{\mcgm}$, therefore 
$$N_{\mcg}[C]=N_{\mcg}(g^n)=N_{\mcg}(\overline{C}). $$
\end{proof}
 
 By Lemma \ref{lem-norm-dehnt}, we have: 
\begin{cor} \label{cor-comm-dehn}
Let  $S$ be an orientable compact surface  with finitely many punctures. Let $\sigma\in \cs$ with vertices $\alpha_1$,...,$\alpha_r$, and $f=\Pi_{i=1}^r T_{\alpha_i}^{n_i} $, with $n_i\in \z-\{0\}$. Then   
$$N_{\mcg}[\langle f\rangle]=N_{\mcg}(f). $$ 
\end{cor}
 
 \subsection{Description of normalizers  } \label{sec-normalizers}
For surfaces with empty boundary, by the Nielsen-Thurston Classification Theorem, infinite order elements are pseudo-Anosov or reducible. 
%%%%%%%%%%%%%%%%%%%%%%%%%%%%%   normalizer of pseudo-Anosov
\begin{thm}\label{normalizer-pseudoA} \cite[Thm. 1]{McCarthy}
Let $S$ be an orientable closed  surface  with finitely many  punctures. Let $f\in \mcg$ be a pseudo-Anosov mapping class. The centralizer $C_{\mcg}(f)$ is a finite extension of an infinite cyclic group. The normalizer, $N_{\mcg}(f)$ is either equal to $C_{\mcg}(f)$ or contains $C_{\mcg}(f)$ as a normal subgroup of index 2.
\end{thm}
 Since $\mcgm$ is torsion free for $m\geq 3$,
for $g\in \mcgm$  a pseudo-Anosov mapping class, we know that $C_{\mcgm}(g)=N_{\mcgm}(g)$ is an infinite cyclic group.\\

 Let $f\in \mcgm$ be  reducible and  $\sigma=\sigma(f)$ with vertices $\alpha_1,...,\alpha_r$. By Lemma \ref{gs=s-reduct}, $C_{\mcg}(f)$ and $N_{\mcg}(f)$ are subgroups of  the stabilizer $\mcg_{\sigma}$. 
Recall that $\mcg_{\sigma}^0$ is  the subgroup of $\mcg_{\sigma}$  that fixes  each $\alpha_i$ with orientation. Denote by $C_{\mcg}(f)^0=C_{\mcg}(f)\cap \mcg_{\sigma}^0$    and $N_{\mcg}(f)^0=N_{\mcg}(f)\cap \mcg_{\sigma}^0$,  the finite index subgroups of $C_{\mcg}(f)$ and $N_{\mcg}(f)$ respectively. 
\begin{prop}\label{prop-centralizer}
Let $S$ be an orientable closed  surface  with finitely many  punctures. Let $f\in \mcgm$ with   $\rhoso(f)=(f_1,...,f_k)$, then
\begin{align}\label{sucesion-centralizer}
\xymatrix{
1\ar[r] & \z^r \ar[r] & C_{\mcg}(f)^0 \ar[r]^(0.4){\rhoso}& \prod_{i=1}^k C_{\Gamma(\widehat{S}_i,\mathcal{Q}_i)}(f_i) \ar[r] & 1 , }
\end{align}
and $C_{\mcg}(f)^0$ has index $\leq 2^k$ in $N_{\mcg}(f)^0$.
\end{prop}  
\begin{proof}
Write $f$ in its canonical form as in (\ref{canonical-descomp}), 
\begin{align*}
  f=\Pi_{i=1}^{k} \eta_{S_i}(\overline{f}_i) \Pi_{j=1}^{r } T_{\alpha_j}^{n_j},
\end{align*}
by Remark \ref{rem-image-pseudo} we have   
\begin{align}\label{eq-theta-i}
(f_1,...,f_k)=(\theta_{S_1}(\overline{f}_1),...,\theta_{S_k}(\overline{f}_k)).
\end{align}
Let $g\in \mcg_{\sigma}^0$, following the method of Theorem \ref{canonical-descom-mod}, since $g$ fixes each subsurface $S_i$  of $\ss$ and each $\alpha_i$ with orientation,  $g$  can be written as    
\begin{align*}
  g=\Pi_{i=1}^{k} \eta_{S_i}(\overline{g}_i) \Pi_{j=1}^{r } T_{\alpha_j}^{m_j},
\end{align*} 
and $\overline{g}_i$ can be reducible, periodic or pseudo-Anosov, for each $i$. In a similar way as  in Remark \ref{rem-image-pseudo},  
$$
\rhoso(g)=(g_1,...,g_k)=(\theta_{S_1}(\overline{g}_1),...,\theta_{S_k}(\overline{g}_k)).
$$
Thus   
\begin{align}\label{norm=cent}
gfg^{-1}=\Pi_{i=1}^{k} \eta_{S_i}(\overline{g}_i)\eta_{S_i}(\overline{f}_i)\eta_{S_i}(\overline{g}_i)^{-1} \Pi_{j=1}^{r } T_{\alpha_j}^{n_j}, 
\end{align}
then 
  \begin{align*} 
& f=gfg^{-1}\\
\text{if and only if}\;\;\;\;  &\eta_{S_i}(\overline{f}_i)=\eta_{S_i}(\overline{g}_i)\eta_{S_i}(\overline{f}_i)\eta_{S_i}(\overline{g}_i)^{-1}, \;\;\; \text{ for all } i,  \\
\text{if and only if}\;\;\;\;  &\overline{f}_i= \overline{g}_i\overline{f}_i\overline{g}_i^{-1}, \;\;\; \text{ for all } i. 
\end{align*}
 The result follows since by the definition of $ \eta_{S_i}$, its kernel is generated by elements of the form $T_{\beta}T^{-1}_{\gamma}$ and any commutator has no Dehn twists about boundary components.

Now, if $\theta_{S_i}(\overline{f}_i)= \theta_{S_i}(\overline{g}_i)\theta_{S_i}(\overline{f}_i)\theta_{S_i}(\overline{g}_i)^{-1}$, then $\overline{g}_i\overline{f}_i \overline{g}_i^{-1}\overline{f}_i^{-1}\in \ker \theta_{S_i} $ which is generated by Dehn twists about boundary components of $S_i$, then we have $\overline{g}_i\overline{f}_i \overline{g}_i^{-1}\overline{f}_i^{-1}=Id$, therefore 
\begin{align*}
 &\overline{f}_i= \overline{g}_i\overline{f}_i\overline{g}_i^{-1}, \;\;\; \text{ for all } i,\\
 \text{if and only if}\;\;\;\; &\theta_{S_i}(\overline{f}_i)= \theta_{S_i}(\overline{g}_i)\theta_{S_i}(\overline{f}_i)\theta_{S_i}(\overline{g}_i)^{-1}\;\; \text{ for all } i, \\
  \text{if and only if}\;\;\;\; &  f_i= g_if_i g_i^{-1} \;\; \text{ for all } i, 
  \end{align*}
which follows by the equality given in  (\ref{eq-theta-i}). \\
Since there are  no restrictions for  $m_j$  with $j\in \{1,...,r\}$, we  conclude (\ref{sucesion-centralizer}). \\
From the equality (\ref{norm=cent}),  if some $n_j\neq 0$, then $C_{\mcg}(f)^0=N_{\mcg}(f)^0$. \\On the other hand, if $n_i=0, \;\text{ for all } i$,   since each $f_i$ is the identity or pseudo-Anosov, in case $f_i$ is pseudo-Anosov,  $C_{\mcgsi}(f_i)$ is a subgroup of index $1$ or $2$ in $N_{\mcgsi}(f_i)$  (Theorem \ref{normalizer-pseudoA}), therefore we conclude that  $C_{\mcg}(f)^0$ has index $\leq 2^k$ in $N_{\mcg}(f)^0$. 
\end{proof}
By Proposition \ref{prop-centralizer} and Theorem \ref{normalizer-pseudoA}, if we rename the subsurfaces $S_i$ as  necessary,   we have: 
\begin{prop}\label{prop-centralizer-vcy}
Let $S$ be an orientable closed  surface  with finitely many  punctures. Let $f\in \mcgm$ with   $\rhoso(f)=(Id_{\widehat{S}_{1}},...,Id_{\widehat{S}_{a}},f_{a+1},...,f_k)$ where $f_{a+1},...,f_k$ are pseudo-Anosov. Then 
\begin{align}
\xymatrix{
1\ar[r] & \z^r \ar[r] & C_{\mcg}(f)^0 \ar[r]^(0.4){\rhos}& \prod_{i=1}^a \Gamma(\widehat{S}_i,\mathcal{Q}_i) \prod_{j=a+1}^k V_j \ar[r] & 1 , }
\end{align}
where $V_j=C_{\Gamma(\widehat{S}_j,\mathcal{Q}_j)}(f_j)$ is virtually cyclic for each $j\in\{a+1,...,k\}$. 
\end{prop}
%%%%%%%

\textsc{Surfaces with boundary}. Suppose that $S$ has  $b\neq 0$  boundary components $\beta_1,...,\beta_b$, let the corking homomorphism
$\theta_{S}\colon \mcg\to \Gamma(\widehat{S}),$
with kernel $\ker(\theta_{S})\simeq \z^b$, generated by Dehn twists about curves isotopic to boundary   components of $S$. Let $\mathcal{R}$ be the set of punctures of $\widehat{S}$ which comes from  the boundary components of $S$. Then $\theta_{S}(\mcg)=\Gamma(\widehat{S},\mathcal{R})$ which is the subgroup of $\Gamma(\widehat{S})$ that fixes pointwise the set $\mathcal{R}$. 

Let $f,g\in \mcg$, suppose that $\theta_S(f)$ and $\theta_S(g)$ commute, then $gfg^{-1}f^{-1}$ is in $\ker\theta_S$, but $gfg^{-1}f^{-1}$ has no Dehn twists about boundary components, so $gfg^{-1}f^{-1}=Id$. Therefore
\begin{align}\label{centraliz-boundary}
\xymatrix{ 1 \ar[r]&\z^b \ar[r]& C_{\mcg}(f) \ar[r]^(0.45){\theta_{S}}& C_{\Gamma(\widehat{S},\mathcal{R})}(\theta_{S}(f)) \ar[r]& 1. 
}
\end{align}
Observe that elements in $\mcg$ leave invariant a regular neighborhood of the boundary $\partial S$.  Then,  if $f$ has a non-zero power of a Dehn twist $T_{\beta_i}$, for any $g\in \mcg$,  $gfg^{-1}$ has the same power of $T_{\beta_i}$, then    
\begin{align}\label{norm-boundary-eq}
N_{\mcg}(f)=C_{\mcg}(f).
\end{align}
 Moreover,  if $f$ has no Dehn twist about curves $\beta_1,...,\beta_b$, then $$gfg^{-1}=f^{\pm 1} \;\;\;\text{iff}\;\;\; \theta_S(g)\theta_S(f)\theta_S(g)^{-1}= \theta_S(f)^{\pm 1}. $$
Then 
 \begin{align}\label{normaliz-boundary}
 \xymatrix{ 1 \ar[r]&\z^b \ar[r]& N_{\Gamma(S)}(f) \ar[r]^(0.45){\theta_{S}}& N_{\Gamma(\widehat{S},\mathcal{R})}(\theta_{S}(f)) \ar[r]& 1 
}.
\end{align}

%%%%%%%%%%%%%%%%%%%%%%%%%%%%%%%%%%%%%%%%%%%%%%%%%%%%%%%%%%%%%%%%%%%%%%%%%%%%%%%%%%%%%%%%%%%%%%%%%%%%%
%%%%%%%%%%%%%%%%%%
%%%%%%%%%%%%%%%%%%
%%%%%%%%%%%%%%%%%%  GEOMETRIC DIMENSION 
%%%%%%%%%%%%%%%%%%
%%%%%%%%%%%%%%%%%%
%%%%%%%%%%%%%%%%%%%%%%%%%%%%%%%%%%%%%%%%%%
%%%%%%%%%%%%%%%%%%%%%%%%%%%%%%%%%%%%%%%%%%%%%%%%%%%%%%%%%%%%%%%%%%%%%%%%%%%%%%%%%%%%%%%%%%%%%%%%%%%%%

\section{Geometric dimension for the family $\vc$ }\label{geom-dimens-mcg}
   
In Section  \ref{sec-geom-dim-mcg-mcgm} we prove that $\gdvc \mcg<\infty$, and in  
  Section \ref{sec-bounds-gdvc}, we will give bounds for $\gdvc \mcgm$ and $\gdvc \mcg$.  

Let $S$ be an orientable compact surface with finitely many punctures  and $\chi(S)<0$. 
It is well-known that the Teichm\"uller space $\mathcal{T}(S)$ is a finite dimensional space which is contractible, on which $\mcg$ acts properly and  it is a model for $\underline{E} \mcg$ by results of Kerckhoff given in \cite{Kerckhoff}.  On the other hand, J. Aramayona and C. Mart\'inez proved in \cite{aramayona} the following:  
\begin{thm}\cite[Cor. 1.3]{aramayona} \label{thm-aramayona}
Let $S$  be an orientable compact surface with finitely many punctures. Then there exist a cocompact model for  $\underline{E}\Gamma(S)$   of dimension equal to the virtual cohomological dimension $vcd(\Gamma(S))$. 
\end{thm} 
And Harer computed $vcd(\Gamma(S))$  in \cite{harer}.    
\begin{thm} \cite[Thm. 4.1]{harer}\label{thm-vcd-harer}
Let $S$ be  an orientable surface surface with genus $g$, $b$ boundary components  and  $n$ punctures. If $2g+b+n>2$, then  
\begin{align*}
vcd(\Gamma(S))= 
        \left \{ \begin{array}{cc}  
             4g+2b+n-4 & \;\; \;\;\;\;\;\;\; 
             \text{if } g >0, \; b+n  >0, 
                  \\   4g-5     & \;\;\; \text{if } n,b=0, 
                  \\ 2b+  n-3     &  \text{if } g=0. 
              \end{array} 
         \right.
\end{align*}
\end{thm}

%%%%%%%%%%%%%%%%%%%%%   GEOMETRIC DIMENSION FOR MCG 
%%%%%%%%%%%%%%%%%%%%%%%%%%%%%%%%%%%%%%%%%%%%%%%%%%%%%%%%%%%
\subsection{Geometric dimension for $\mcg$}\label{sec-geom-dim-mcg-mcgm}

 We will use the same notation of Section \ref{section-build}.
   We will prove that there exist finite dimensional models for $\underline{E} N_{\mcg}[C]$ and   $E_{\g[C]} N_{\mcg}[C]$ and a uniform bound on $\gd_{\g[C]}N_{\Gamma}[C]$ for any $[C]\in [\C_{\mcg}]$, with these results,  Theo\-rem \ref{thm-luck-weierm}  and the fact that $\gdf \mcg$ is finite, we have:
\begin{thm}\label{thm-gdf-mcg}
Let $S$ be an orientable compact surface with finitely many punctures  and $\chi(S)<0$. Then $\gdvc \mcg< \infty$, that is, the  mapping class group $\mcg$ admits a finite dimensional model for $\underline{\underline{E}} \mcg$. 
\end{thm}

For the  proof of Theorem \ref{thm-gdf-mcg} we need the following results. 
\begin{prop}\cite[Prop. 4]{Daniel-Leary} \label{prop-Daniel-Leary}
Let $G$ be an infinite virtually cyclic group, then there is model for  $\efin$ with finitely many orbits of cells which is homeomorphic to the real line.  
\end{prop}
\begin{thm}\cite[Thm. 5.16]{luck}\label{seq-efin}
Let $1 \to H \to G \to K\to 1$ be an exact sequence of groups.  Suppose that $H$ has the property that for any group $\tilde{H}$ which contains $H$ as subgroup of finite index, $\gdf \tilde{H}\leq n$. If $\gdf K\leq k$, then $\gdf G\leq n+k$.   
\end{thm}
In  \cite[Ex. 5.26]{luck} L\"uck shows that virtually poly-cyclic groups satisfies the condition about $H$ in Theorem 
\ref{seq-efin}, in particular $\z^n$ satisfies  such condition. 
\begin{thm}\cite[Thm. 2.4]{luck-type}\label{finite-index-gd}
Suppose  $H\subseteq G$ is a subgroup of finite index $n$, then $\gdf G \leq  \gdf H \cdot n\;$ and $\;\gdvc G\leq  \gdvc H \cdot n $. 
 \end{thm}
\begin{prop}\cite[Lem. 4.3]{Guido} \label{prop-guido}
Let $1\to \z^n \to G \to F$ be an exact sequence of groups with $F$ finite. Then $G$ admits an $n$-dimensional cocompact $\efin$ homeomorphic to $\r^n$, with $G$ acting by affine maps. 
\end{prop}
\begin{rem}\label{rem-w-ec}
Let $[C]\in [\C_{\mcg}]$. From Proposition \ref{prop-comm-mcgm}, we could assume that $C$ is a cyclic  maximal subgroup  in $\C_{\mcgm}$ and 
$N_{\mcg}[C]=N_{\mcg}(C)$.  
Let $W_{\mcg}({C})=N_{\mcg}({C})/{C}$ and $p\colon N_{\mcg}({C})\to W_{\mcg}({C})$,  the projection.\\
From Theorem \ref{thm-esp-clas-defi}, a model for $\underline{E} W_{\mcg}({C})$ with the $N_{\mcg}[C]$-action induced from the projection $p$ is a model for $E_{\g[C]} N_{\mcg}[C]$.  \\
 Then it is sufficient to consider  models for  $\underline{E} N_{\mcg}(C)$ and   $\underline{E}W_{\mcg}(C)$, of maximal infinite cyclic subgroups $C$ in  $\C_{\mcgm}$.  
 \end{rem} 

 %

%%%%%%%%%%%%%%%%%%%%
%%%%%%%%%%%%%%%%%%%%    PROOF
%%%%%%%%%%%%%%%%%%%%
\textit{Proof of Theorem} \ref{thm-gdf-mcg}: Suppose that the surface $S$ has genus $g$, $b$ boundary components and $n$ punctures. \\
\textbf{Part 1.} \textit{For any $[C]\in [\C_{\mcg}]$, $\gdf N_{\mcg}[C]$ is finite.}
We may assume that $C\in \C_{\mcgm}$, it follows that $N_{\mcg}[C]=N_{\mcg}({C})$. 
 Since $N_{\mcg}({C})\leq \mcg$ from the properties given in (\ref{gd-subg}) and Theorem \ref{thm-aramayona}, we conclude that $\gdf N_{\mcg}[C]\leq \gdf \mcg\leq vcd \mcg$, which is finite. \\ 
\textbf{Part 2.}
  \textit{ We will prove that there exist $z\in \z$, such that for any $[C]\in [\C_{\mcg}]$,  $\gd_{\g[C]} N_{\mcg}[C]\leq z.$}\\
(I) \emph{The surface $S$ has empty boundary. } Let $[C]\in \C_{\mcg}$,  with $C=\langle f\rangle$.  By the Nielsen-Thurston classification Theorem, $f$ is  either,  a pseudo-Anosov class or  a reducible element.\\
(a) If $f$ is pseudo-Anosov, then $N_{\mcg}[C]=N_{\mcg}(f)$ is virtually cyclic and  $\g [C]$ is the family of all subgroups of $N_{\mcg}(f)$, hence a point is a model for $E_{\g [C]}N_{\mcg}[C]$, therefore $\gd_{\g[C]}N_{\mcg}[C]=0$.   \\
(b) If $f$ is reducible and   $f=\Pi_{i=1}^r T_{\alpha_i}^{n_i}$, with $n_i\in \z-\{0\}$, where $\alpha_1,...,\alpha_r$ are the vertices of $\sigma=\sigma(f)$. By Corollary \ref{cor-comm-dehn}
,   $N_{\mcg}[C]=N_{\mcg}(f)$ and we can suppose that $gdc\{n_1,...,n_k\}=1$. \\
Following the same idea as in Remark \ref{rem-w-ec},  a model for $\underline{E} W_{\mcg}(f)$ with the induced action of the projection $N_{\mcg}(f)\to W_{\mcg}(f)$ is a model for $E_{\g[C]}N_{\mcg}[C]$. We will prove that $\gdf W_{\mcg}(f)$ is finite. \\
 Note that any element  $g\in \mcg_{\sigma}^0$ commutes with $f$, because  $g$ fixes each class $\alpha_i$, then $\mcg_{\sigma}^0\subseteq N_{\mcg}(f)$. On the other hand, we have    $N_{\mcg}(f)\subseteq \mcg_{\sigma}$, therefore $\mcg_{\sigma}^0$ is a finite index subgroup of $N_{\mcg}(f)$  and it is normal, then 
\begin{align*}
\xymatrix{
1\ar[r] & \mcg_{\sigma}^0 \ar[r] & N_{\mcg}(f)\ar[r]& B  \ar[r] & 1 , }
\end{align*}
where $B$ is finite, we will obtain a uniform bound for the order of $B$. Since $f\in \mcg_{\sigma}^0 $, we have 
\begin{align*}
\xymatrix{
1\ar[r] & \mcg_{\sigma}^0/\langle f\rangle \ar[r] & N_{\mcg}(f)/\langle f\rangle \ar[r]& B  \ar[r] & 1 . }
\end{align*}
Since the index $[\mcg_{\sigma}:\mcg_{\sigma}^0]\leq (2r)!$, then $|B|\leq (2r!)$.  By Theorem \ref{finite-index-gd}, 
\begin{align}
\gdf W_{\mcg}(f)& \leq \gdf(\mcg_{\sigma}^0/\langle f\rangle)\cdot |B| \nonumber
\\ & \leq  \gdf(\mcg_{\sigma}^0/\langle f\rangle) \cdot (2r)!. \label{ineq-w-t}
\end{align}  
From (\ref{cutting-homom-0}), since $f\in \ker(\rhoso)$ we have
\begin{align*}
\xymatrix{
1\ar[r] & \langle T_{\alpha_1},....,T_{\alpha_r}\rangle /\langle f\rangle \ar[r] & \mcg_{\sigma}^0/\langle f\rangle \ar[r]& \prod_{i=1}^k \Gamma(\widehat{S}_i,\mathcal{Q}_i)  \ar[r] & 1 , }
\end{align*}
since $\langle T_{\alpha_1},....,T_{\alpha_r}\rangle\simeq \z^r$  and $f$ is identified with  the point $(n_1,...,n_r)\in \z^r$ via that isomorphism, then the quotient $\z^r/\langle(n_1,...,n_r)\rangle\simeq \z^{r-1}$, because $gdc\{n_1,...,n_r\}=1$.  Then    
\begin{align*}
\xymatrix{
1\ar[r] & \z^{r-1} \ar[r] & \mcg_{\sigma}^0/\langle f\rangle \ar[r]& \prod_{i=1}^a \Gamma(\widehat{S}_i,\mathcal{Q}_i) \ar[r] & 1 , }
\end{align*}
we apply Theorem \ref{seq-efin}, the properties given in   (\ref{efin-product}) and (\ref{gd-subg}), and Theorem \ref{thm-aramayona}
 to conclude  that 
\begin{align}
\gdf (\mcg_{\sigma}^0/\langle f\rangle) &\leq (r-1)+ \gdf(\Pi_{i=1}^a \Gamma(\widehat{S}_i,\mathcal{Q}_i) )  \nonumber \\
& \leq (r-1)+ \sum_{i=1}^a \gdf(\Gamma(\widehat{S}_i,\mathcal{Q}_i)) \nonumber \\
& \leq (r-1)+ \sum_{i=1}^a vcd(\Gamma(\widehat{S}_i)). \label{eq-vcd-Si}
\end{align}
Observe that each $\widehat{S}_i$ has at least one puncture, no boundary components and  negative Euler characteristic, thus from Theorem \ref{thm-vcd-harer} we can see  that  $vcd(\Gamma(\widehat{S}_i))\leq -2\chi(\widehat{S}_i)$, then  
\begin{align}
\gdf (\mcg_{\sigma}^0/\langle f\rangle) &\leq  (r-1) + \sum_{i=1}^a (-2\chi(\widehat{S}_i)) \nonumber
\\
& \leq  (r-1)+ (-2 \chi(S)),\label{ineq-wt2}
\end{align}
from inequalities (\ref{ineq-w-t}) and (\ref{ineq-wt2}), we conclude 
\begin{align}
\gdf W_{\mcg}(f)& \leq  (-2 \chi(S)+ r-1)\cdot (2r)!. 
\end{align}
By Remark \ref{rem-euler-k-r}
$r\leq \frac{-3\chi(S)-n}{2}$, therefore 
\begin{align}\label{ineq-partb}
\gdf W_{\mcg}(f)& \leq  (-5\chi(S)-n) (-3\chi(S)-n)!. 
\end{align}
Note that the bound only depends on the surface.
 \\
%%%%%%%%
%%%%%%%%
(c) Now suppose that $f$ is reducible,  $\sigma=\sigma(f)$ has vertices $\alpha_1,...,\alpha_r$,  and $\rho(f)$ is not trivial.  Since $N_{\mcg}[C]=N_{\mcg}(f^n)$, for some $n\neq0$ such that $f^n\in \mcgm$,   then  we may assume that $f\in \mcgm$ and $C=\langle f\rangle $ is maximal in $\C_{\mcgm}$. We  will apply Remark  \ref{rem-w-ec} again.   \\
Note that $C_{\mcg}(f)^0 \unlhd N_{\mcg}(f)$ is of finite index,  then
\begin{align*}
\xymatrix{
1\ar[r] & C_{\mcg}(f)^0 \ar[r] & N_{\mcg} (f)   \ar[r] & F
\ar[r] & 1 , }
\end{align*}
with $|F|\leq 2^k((2r)!)$, this follows by Proposition \ref{prop-centralizer} and because the index  $[\mcg_{\sigma}:\mcg_{\sigma}^0]\leq (2r)!$.  Since $f\in \mcgm$, then $f\in C_{\mcg}(f)^0$, and so, 
\begin{align*}
\xymatrix{
1\ar[r] & C_{\mcg}(f)^0/\langle f\rangle \ar[r] & N_{\mcg} (f)/\langle f\rangle   \ar[r] & F
\ar[r] & 1 . }
\end{align*}
By Theorem \ref{finite-index-gd} 
\begin{align}\label{ineq-red}
\gdf (W_{\mcg}(C))\leq \gdf (C_{\mcg}(f)^0/\langle f\rangle) \cdot 2^k((2r)!). 
\end{align}
Let $\rhoso(f)=(f_1,...,f_k)$ and  we  rename the  $\widehat{S}_i$ such that $f_i=Id_{\Gamma(\widehat{S_i})} $ for $i\in \{1,...,a\}$ and $f_j\in \Gamma(\widehat{S_j},\mathcal{Q}_j)$ is pseudo-Anosov for $j\in \{a+1,..,k\}$, then by Proposition \ref{prop-centralizer-vcy},   
\begin{align*}
\xymatrix{
1\ar[r] & \langle T_{\alpha_1},....,T_{\alpha_r}\rangle \ar[r] & C_{\mcg}(f)^0 \ar[r]^(0.35){\rhoso}& \prod_{i=1}^a \Gamma(\widehat{S}_i,\mathcal{Q}_i) \prod_{j=a+1}^k V_j  \ar[r] & 1 , }
\end{align*}
where $V_j=C_{\Gamma(\widehat{S_j},\mathcal{Q}_j)}(f_j)$ is  virtually cyclic for each $j$. \\
 Denote  the group $\prod_{i=1}^a \Gamma(\widehat{S}_i,\mathcal{Q}_i) \prod_{j=a+1}^k V_j $ by $\Delta$, then we have the following homomorphism
\begin{align*}
\psi \colon C_{\mcg}(f)^0 /\langle f \rangle  &\to \Delta/\langle \rhoso(f)\rangle 
\\     g\langle f\rangle &\mapsto \rhoso(g)\langle \rhoso(f)\rangle,
\end{align*}
which is well-defined because $\rhoso(\langle f\rangle)= \langle \rhoso(f)\rangle$,  $\psi$ is a homomorphism and  since $\rhoso$  is onto, then $\psi$ is onto too. Thus 
$$
\ker\psi= \rhoso^{-1}(\langle \rhoso(f)\rangle )/\langle f\rangle,
$$
is a free abelian subgroup isomorphic to $\z^r$. Moreover,
\begin{align*}
\frac{\Delta}{\langle \rhoso(f)\rangle }&=\frac{\Pi_{i=1}^a \Gamma(\widehat{S}_i,\mathcal{Q}_i)\Pi_{j=a+1}^k V_j } {\langle(Id_{\Gamma(\widehat{S_1})},...,Id_{\Gamma(\widehat{S_a})}, f_{a+1},...,f_{k} )\rangle } 
\\
&= \Pi_{i=1}^a \Gamma(\widehat{S}_i,\mathcal{Q}_i) \times \frac{\Pi_{j=a+1}^k V_j}{\langle(f_{a+1},...,f_{k} )\rangle}.
\end{align*}
Since $\langle f_j\rangle\unlhd V_j$ is of finite index,  $\text{ for all } j$, if $\check{f}=(f_{a+1},...,f_k)$,  
\begin{align}\label{vj}
\xymatrix{
1\ar[r] & \Pi_{j=a+1}^k \langle f_j\rangle/\langle \check{f}\rangle  \ar[r] & \Pi_{i=a+1}^k  V_j/ \langle \check{f}\rangle  \ar[r]& \Pi_{i=a+1}^k F_j \ar[r] & 1 , }
\end{align} 
with $F_j$ finite $\text{ for all } j$, thus from Proposition \ref{prop-guido} and (\ref{vj}) we have that   $\gdf (\Pi_{i=a+1}^k  V_j/ \langle \check{f}\rangle) \leq k-a-1$. Thus applying Theorem \ref{seq-efin}  to the exact sequence given by $\psi$,  the properties given in   (\ref{efin-product}) and (\ref{gd-subg}), and Theorem \ref{thm-aramayona}, we have: 
\begin{align}
\gdf (C_{\mcg}^0/\langle f\rangle)& \leq r+ [\gdf (\Pi_{i=1}^a\Gamma(\widehat{S}_i,\mathcal{Q}_i)) + (k-a-1)] \nonumber \\ 
&\leq r+ [\gdf (\Pi_{i=1}^a\Gamma(\widehat{S}_i)) + (k-a-1)]\nonumber  \\
&\leq r+ [\sum_{i=1}^a \gdf (\Gamma(\widehat{S}_i)) + (k-a-1)]\nonumber \\
& \leq r+ [\sum_{i=1}^a vcd (\Gamma(\widehat{S}_i)) + (k-a-1)] \nonumber \\
& \leq r+ [\sum_{i=1}^a (-2\chi(\widehat{S}_i)) + (k-a-1)] \nonumber \\
&\leq r+ [(-2\chi(S)) + (k-a-1)] \nonumber \\
&\leq \frac{-3\chi(S)-n}{2}+ (-3\chi(S)-1), \label{eq-c0f}
\end{align}
the last inequality (\ref{eq-c0f}) holds because $k\leq -\chi(S)$ (see Remark \ref{rem-euler-k-r}) and $a\geq 0$. From (\ref{ineq-red}) and (\ref{eq-c0f}) we conclude that
\begin{align}\label{ineq-partc}
\gdf (W_{\mcg}(C))\leq (-6\chi(S)-n)  ((-3\chi(S)-n)!). 
\end{align}
Note that the bound only depends on the surface. 
From (a) and the inequalities (\ref{ineq-partc}) 
and (\ref{ineq-partb}), for any $[C]\in [\C_{\mcg}]$,  
\begin{align}
\gd_{\g[C]} N_{\mcg}[C]\leq (-6\chi(S)-n) ((-3\chi(S)-n)!).
\end{align} \\
(II) \emph{The surface $S$ has non-empty boundary. } Following in a similar way as in Part (I), from  (\ref{centraliz-boundary}),(\ref{norm-boundary-eq})  and (\ref{normaliz-boundary}), we   conclude that there exist $z\in \z$, such that for any $[C]\in [\C_{\mcg}]$,  $\gd_{\g[C]} N_{\mcg}[C]\leq z$.   
\qed
\\

\subsection{Bounds for geometric dimension} \label{sec-bounds-gdvc}

 We will prove that $\mcgm$ satisfies the following property, which we will use to give a bound for $\gdvc \mcgm$. 
  \begin{defi}
A group $G$ satisfies $Max_{\vci_G}$ if every subgroup $H\in \vci_G$ is contained in a unique $H_{max}\in \vci_G$ which is maximal in $\vci_G$. \end{defi}

We  follow the same notation as in  Section \ref{sec-prelim} and
\ref{sec-stabilizer}. The homomorphism $\rhos$  is given in Section \ref{section-stabilizers},  and $\rhosm=\rhos|_{\mcgm}$.
By  \cite[Thm. 1.2]{ivanov}, elements in $\mcgm_{\sigma}$ do not rearrange the components of $S_{\sigma}$ and fix each curve of $\sigma$, 
 then $\mcgm_{\sigma}\subseteq \mcg_{\sigma}^0$.  Note that if $g\in \mcgm_{\sigma}$  and 
  $\rhosm(g)=(g_1,...,g_k)$, by definition of $\rhosm$ and because $g$ is pure, each $g_i$ is pure, then
 the image $\rhosm(\mcgm_{\sigma})$ is a pure subgroup of $\prod_{i=1}^k \mcgsiqi$. Let  $\Gamma_i$ be the projection of the image $\rhosm(\mcgm_{\sigma})$ over  $\mcgsiqi$ for each $i$, then each $\Gamma_i$ is torsion free and we have  
 \begin{align}\label{cutting-hom-m}
\rhosm \colon \mcgm_{\sigma }\to \Pi_{i=1}^k \Gamma_i,
\end{align}
where  $\ker(\rhosm)$ is a free abelian subgroup of $\ker(\rhos)\simeq \z^r$, observe that $\ker(\rhosm)\simeq \z^r$ as $\mcgm$ is of finite index and 
$\ker(\rhosm)$ is of finite index in $ker(\rho_{\sigma})$ . As a reference, see \cite[Sec. 7.5]{ivanov}. 

\begin{lem}  \cite[Lem. 8.7]{ivanov} \label{lem-abelian}
Suppose that $S$ has empty boundary. Let $G$ be a subgroup of $\mcgm$, $m\geq 3$. Let $\sigma=\sigma(G)$, and suppose that $\rhos(G)=\Pi_{i=1}^k G_i$, where  $G_i$ denotes the projection of $\rhosm(G)$ over $\mcgsiqi$ for each $i$. The group   $G$ is abelian if and only if each $G_i$ is either trivial or an infinite cyclic group.  
\end{lem}

\begin{prop}\label{prop-modm-maximality}
Let $S$ be an orientable compact surface with finitely many punctures  and $\chi(S)<0$. 
 Let $m\geq 3$, then the group $\mcgm $ satisfies   property $Max_{\vci_{\mcgm}}$. 
 \end{prop}
\begin{proof}
\textbf{Case I}: Suppose that $S$ has empty boundary.    
Since $\mcgm$  is torsion free for $m\geq 3$,   
then $\vci_{\mcgm}=\C_{\mcg_m}$ is  the set  of infinite cyclic subgroups of $\mcgm$.\\
It is well-known that for surfaces $S$ with empty boundary, periodic elements of $\mcg$ are of
 finite order.  By the Nielsen-Thurston classification Theorem,
 each element of $\mcgm-\{Id\}$ is either reducible or pseudo-Anosov.\\
Let  $H=\langle f \rangle\in \C_{\mcg_m}$. Observe that if $\langle k\rangle=K\in \C_{\mcg_m}$ and $H\subseteq K$, then $f^n=k$  for some $n\in\z-\{0\}$. By Lemma \ref{lem-n(f)=n(fk)},
$C_{\mcgm}(K)= C_{\mcgm}(H)$, thus $K\subseteq C_{\mcgm}(H)$. Actually, we will prove that $K$ lies in a free abelian  subgroup of $\mcgm$ (which does not depends on $K$ but only on $H$), 
 then $H$ must be contained in an unique maximal subgroup of $\C_{\mcgm}$. 
 If $f$ is pseudo-Anosov, by Theorem \ref{normalizer-pseudoA},  
   $C_{\mcgm}(H)\in \C_{\mcgm}$  and it is the unique maximal subgroup of $\C_{\mcgm}$ containing $H$. 
   
   On the other hand,
suppose that $f$ is reducible and let $\sigma=\sigma(f)$ be its canonical reduction system,   by Lemma \ref{gs=s-reduct},   $C_{\mcgm}(H)\subseteq \mcgm_{\sigma}$.  Suppose that $\sigma$ has vertices  $\alpha_1,...,\alpha_r$, let  $\ss=S_1\cup \cdots S_k$ and $\rhosm$ as in (\ref{cutting-hom-m}),
 \begin{align*}
\rhosm \colon \mcgm_{\sigma }\to \Pi_{i=1}^k \Gamma_i, 
\end{align*}
where $\Gamma_i \subset \Gamma(\widehat{S}_i,\mathcal{Q}_i)$ is torsion free  for each $i$ and $\ker(\rhosm)\simeq \z^s$ is a  free abelian subgroup of $\langle T_{\alpha_1},...,T_{\alpha_r}\rangle \simeq \z^r$. Note that each $T_{\alpha_i}$ commutes with $f$, if $\rhosm(f)=(f_1,...,f_k)$,  then 
\begin{align*}
 \xymatrix{1 \ar[r]& \z^s \ar[r]& C_{\mcgm}(f) \ar[r]^(0.45){\rhosm}& \Pi_{i=1}^k C_{\Gamma_i}(f_i)\ar[r] &1 
},
\end{align*}
furthermore each $f_i$ is either the identity or pseudo-Anosov, see Remark  \ref{rem-image-pseudo}. Let $\rhosm(k)=(k_1,...,k_r)$, then we have that
\begin{itemize}
\item[(a)] $\text{for all } i,\;\; \; k_if_i=f_ik_i$,   
\item[(b)] $\text{for all } i\;$,  $f_i$ is an $n$-th root of $k_i$ because $f$ is an $n$-th root of $k$,
\item[(c)] if $f_j=Id$ for some $j$, then $k_j=Id$, because each $\Gamma_j$ is torsion free. 
\end{itemize}
Let $L=\{l_1,...,l_d\}\subseteq\{1,...,k\}$, such that  $l_i\in L$ if only if $f_{l_i}$ is pseudo-Anosov. By Theorem \ref{normalizer-pseudoA}, $C_{\Gamma_{l_i}}(f_{l_i})$ is an infinite cyclic subgroup for each $l_i\in L$. We regard $\Pi_{j=1}^d C_{\Gamma_{l_j}}(f_{l_j})$
  as subgroup of $\Pi_{i=1}^k C_{\Gamma_i}(f_i)$, let 
$$G=\rhosm^{-1}(\Pi_{j=1}^d C_{\Gamma_{l_j}}(f_{l_j}))\subseteq \mcgm_{\sigma},$$ 
then we have 
\begin{align*}
 \xymatrix{1 \ar[r]& \z^s \ar[r]& G \ar[r]^(0.35){\rhosm}& \Pi_{j=1}^d C_{\Gamma_{l_j}}(f_{l_j})\ar[r]& 1 
},
\end{align*}
by Lemma \ref{lem-abelian}, we conclude that $G$ is a free abelian subgroup and by cons\-truction, we have that  $K\subseteq G$. Since  $K$ was taken arbitrarily,  we conclude that there exists a unique maximal subgroup $H_{max}\in \C_{\mcgm}$ containing $H$.     
\textbf{Case II.} Suppose that $S$ has non empty boundary and that $S$ has $b\neq 0$ connected boundary components $\beta_1,...,\beta_b$.  Let $\theta_S\colon \mcg \to \Gamma(\widehat{S})$, be  the corking homomorphism as in (\ref{corking-h}),  with 
$$
\ker( \theta_S)=\langle T_{\beta_1},....,T_{\beta_b}\rangle \simeq \z^b.
$$
  By definition of $\theta_S$ we have that $\theta_S(\Gamma_m(S))\subseteq \Gamma_m(\widehat{S})$.     
Since each $T_{\beta_i}$ acts trivially on $H_1(S,\z)$,  $\ker( \theta_S) \subseteq \mcgm$, then $\ker(\theta_S|_{\mcgm})=\ker(\theta_S)$.  Then we have the following,   
\begin{align} \label{cork-m}
\xymatrix{
   1\ar[r]& \z^b \ar[r]& \mcgm \ar[r]^{\theta_{S,m}} &\Gamma_m(\widehat{S}).
}
   \end{align}
It is well known that $\mcg$ is  torsion free when $S$ has non-empty boundary, then $\C_{\mcgm}=\vc_{\mcgm} -\fin_{\mcgm}$ is the set of infinite virtually cyclic subgroups of $\mcgm$.\\ 
Let $A=\langle x \rangle \in \C_{\mcgm}$, if $B=\langle y\rangle \in \C_{\mcgm}$ is such that $A\subseteq B$, then $\theta_{S}(A)=\langle \theta_{S}(x)\rangle \subseteq \langle \theta_{S}(y)\rangle=\theta_{S}(B)$. As in CASE I, we will prove that $B$ lies in a free abelian subgroup of $\mcgm$. \\
By CASE I, we have that there exists a free abelian subgroup 
$\overline{G}\subset \Gamma_m(\widehat{S})$ such that if $K\in \C_{\Gamma_m(\widehat{S})}$ and $\theta_{S}(A)\subseteq K$, then $K\subseteq \overline{G}$. Let $G=\theta_{S,m}(\mcgm)\cap \overline{G}$, note that $G$ is a free abelian subgroup and $B\subseteq \theta_{S,m}^{-1}(G)$, because $y\in  \theta_{S,m}^{-1}(\theta(y) )$. 
\\
Let $g_1,g_2\in G$, $\tilde{g}_1\in \theta_S^{-1}(g_1)$ and $\tilde{g}_2\in \theta_S^{-1}(g_2)$, since $g_1$ and $g_2$ 
commute, by definition of $\theta_S$, it follows that $\tilde{g}_1$ and  $\tilde{g}_1$ must commute, therefore $\theta^{-1}(G)$  is a free abelian subgroup. Since $B$ was taken arbitrarily,   we  conclude that $A$ is contained in an unique maximal subgroup $A_{max}\in \C_{\mcgm}$. 
\\
Note that the free abelian subgroup does not depends on $B$ but only on $A$.
\end{proof}
We will apply the following Theorem.
\begin{thm}\cite[Thm. 5.8]{luck-weiermann}\label{thm-dimens-maximal}
Let $G$ be a group satisfying $Max_{\vci_G}$. Suppose  we know  that $\; \gdvc G < \infty$, then
  \begin{align}
    \gdvc G \leq \gdf G +1 .
  \end{align}
\end{thm} 
\begin{thm} \label{thm-cotas-mcg-mcgm}
Let $S$ be an orientable compact surface with finitely many punctures  and $\chi(S)<0$. 
 Let $m\geq 3$, then 
\begin{enumerate}
\item[(1)] $\gdvc \Gamma_m(S) \leq vcd( \Gamma(S)) +1 $;
\item[(2)] Let $[\Gamma(S):\Gamma_m(S) ]$ be the index  of  $\Gamma_m(S)$ in $\Gamma(S)$, then
\begin{align*}\gdvc \Gamma(S) &\leq  [\Gamma(S):\Gamma_m(S) ] \cdot \gdvc \Gamma_m(S)\\ &\leq [\Gamma(S):\Gamma_m(S) ]  \cdot (vcd(\Gamma(S) ) +1 ) .\end{align*} 
\end{enumerate}
Where $vcd(\Gamma(S))$ is the virtual cohomological dimension of $\Gamma(S)$.
\end{thm}
%%%%%%%%%%%%h%%%%%%%%%%%%%%%%%%%%%%%%%%%%%%%%%%%%%%%%%%
\begin{proof} 
By Theorem \ref{thm-dimens-maximal}, Proposition \ref{prop-modm-maximality} and  Theorem \ref{thm-gdf-mcg}, we conclude (1)  and applying Theorem \ref{finite-index-gd} we conclude (2). 
\end{proof}

%%%%%%%%%%%%%%%%%%%%%%%%%%%%   BRAID GROUPS
%%%%%%%%%%%%%%%%%%%%%%%%%%%%%%%%%%%%%%%%%%%%%%%%%%%%%%%%%%%
 
%\subsection{Surface braid groups} 

\end{document}